\documentclass[11pt]{article}
\usepackage[utf8]{inputenc}
\usepackage[T1]{fontenc}
\usepackage[english]{babel}
\usepackage{indentfirst}
\usepackage{latexsym}
\usepackage[usenames]{color}
\usepackage[colorlinks=true, bookmarksnumbered=true, bookmarksopen=true,
bookmarksopenlevel=3, pdfstartview=FitH, linkcolor=cyan, pdfmenubar=true,
pdftoolbar=true, bookmarks=true,citecolor=cyan, urlcolor=magenta,
filecolor=cyan,menucolor=black,plainpages=false,pdfpagelabels]{hyperref}
\usepackage[lmargin=2cm,tmargin=2cm,bmargin=2cm,rmargin=2cm]{geometry}
\usepackage{amsbsy, amsmath,amsfonts,amssymb,amsthm}
\usepackage{times}
\usepackage{graphicx}
\usepackage{amsmath}
\usepackage{amsfonts}
\usepackage{amssymb}
\newcommand{\N}{\mathbb{N}}
\newcommand{\C}{\mathbb{C}}

\newcommand{\R}{\mathbb{R}}

\newcommand{\G}{\mathbb{G}}
\newcommand{\Z}{\mathbb{Z}}
\renewcommand{\mod}{\operatorname{mod}}

\newcommand{\abs}[1]{\left\lvert#1\right\rvert}
\newcommand{\norm}[1]{\left\Vert#1\right\Vert}
\newcommand{\dual}[2]{\left\langle{#1},{#2}\right\rangle}

\newtheorem{theorem}{Theorem}[section]

\newtheorem{proposition}[theorem]{Proposition}

\newtheorem{lemma}[theorem]{Lemma}
\newtheorem{remark}[theorem]{Remark}
\newtheorem{example}[theorem]{Example}
\numberwithin{equation}{section}

\begin{document}

\title{Estimates for entropy numbers of multiplier operators of multiple series}
\author{Sergio A. Córdoba\thanks{email: sergiocordoba@unicauca.edu.co}, 
J\'essica Milar\'e\thanks{email: jessymilare@gmail.com} and
Sergio A. Tozoni\thanks{email: tozoni@unicamp.br}}
\date{}
\maketitle

\bigskip

\noindent
$^{\ast}$ \quad {\em Departamento de Vías y Transporte, Universidad del Cauca, Calle 5 Nº 4-70, Popayán - Cauca, CEP 190003, Colombia}\\
$^{\dagger, \ddagger}$ \quad {\em Instituto de Matem\'{a}tica, Universidade Estadual de Campinas, 
Rua S\'{e}rgio Buarque de Holanda 651, Campinas-SP, CEP 13083-859, Brazil}

\bigskip

\begin{abstract}
 The asymptotic behavior for entropy numbers of general Fourier multiplier operators of multiple series with respect to an abstract complete orthonormal system $\{\phi_{\textup{\textbf{m}}}\}_{\textup{\textbf{m}}\in \N^{d}_{0}}$ on a probability space and bounded in $L^{\infty}$, is studied. The orthonormal system 
can be of the type 
$\phi_{\textup{\textbf{m}}}(\textup{\textbf{x}})=\phi_{m_{1}}^{(1)}(x_{1})\cdots \phi_{m_{d}}^{(d)}(x_{d})$,
where each $\left\{\phi_{l}^{(j)}\right\}_{l\in \N_{0}}$ is an orthonormal system, that can be different for each $j$, for example, it can be a Vilenkin system, a Walsh system on a real sphere or the trigonometric system on the unit circle.
General upper and lower bounds for the entropy numbers are established by using Levy means of norms constructed using the orthonormal system. These results are applied to get upper and lower bounds for entropy numbers of specific multiplier operators, 
which generate, in particular cases, sets of finitely and infinitely differentiable functions, in the usual sense and in the dyadic sense.
It is shown that these estimates have order sharp in various important cases.

\bigskip



\end{abstract}

\section{Introduction}


In \cite{Bordin, Kushpel11111, Kushpell, Stabile}, the asymptotic behavior of entropy numbers of multiplier operators was studied, using different techniques. Estimates were obtained for entropy numbers of sets of finitely and infinitely differentiable functions and sets of analytic functions, on homogeneous spaces and on the torus. In this paper, we continue these studies considering now general Fourier multiplier operators of multiple series with respect to an abstract complete orthonormal system.




Recently, much attention has been devoted to the study of entropy numbers of different sets of functions. This problem has a long history and some fundamental problems in this area are still open. 
Estimates for entropy numbers of different operators and embeddings between function spaces of the types Besov and Sobolev with mixed smoothness into Lebesgue have been studied \cite{ Kushpel2021,Mayera,Mieth,Temlyakov1,Wang}.
In the papers \cite{Kushpell, Stabile} sharp estimates were obtained for entropy numbers of sets of infinitely differentiable functions and of analytic functions, on two-points homogeneous spaces and on the torus, in several cases.
We do not know other papers where studies of this type have been carried out.
Our approach is based on estimates for Levy means of norms constructed using an orthonormal system of multiple functions on a probability space.
To prove the results for general multiplier operators we make use of deep results from Banach space geometry in Pisier \cite{Pisier} and in 
the paper of Pajor and Tomczak-Jaegermann \cite{Pajor}.
For our applications we need results on the amount of points of $\mathbb{Z}^d$ contained in a closed Euclidean ball of $\mathbb{R}^d$.
We consider the study on the classical Gauss circle problem in Huxley \cite{Hu} and the study on the three-dimensional sphere problem in  
Heath-Brown \cite{HB}.
A quality of this method is to enable the study of entropy numbers of sets of infinitely differentiable and analytic functions, allowing to obtain sharp estimates in terms of order, in several cases, which has not been possible, until now, with the use of other methods.
In particular we get lower and upper bounds for entropy numbers of classes of finitely and infinitely differentiable functions, in the usual sense and in the dyadic sense. The estimates in our applications have order sharp in various important cases.

Consider two Banach spaces $X$ and $Y$. The norm of $X$ will be denoted by $\norm{\cdot}$ or $\norm{\cdot}_{X}$ and the closed unit ball $\lbrace x\in X: \norm{x}\leq 1\rbrace$  by $B_{X}$. Let $A$ be a compact subset of $X$. The $n$-th entropy number $e_{n}(A,X)$ is defined as the infimum of all positive $\epsilon$ such that there exist $x_{1},\ldots,x_{2^{n-1}}$ in $X$ satisfying $ A\subset \cup_{k=1}^{2^{n-1}}(x_{k}+\epsilon B_{X})$, that is,
\[
e_{n}(A,X)=\inf \biggl\{ \epsilon>0: A\subset \bigcup_{k=1}^{2^{n-1}}\bigl(x_{k}+\epsilon B_{X}\bigl) \textup{ for some } x_{1},\ldots,x_{2^{n-1}}\in X \biggl\}.
\]
If $T\in\mathcal{L}(X,Y)$ is a compact operator, the $n$-th entropy number $e_{n}(T)$ is defined as $e_{n}(T)=e_{n}\left(T(B_{X}),Y\right)$.

Let $\N$ denote the set of positive integers and $\N_{0}$ the set of non-negative integers. For $d\in\N$ and $\textup{\textbf{k}}\in \N^{d}_{0}$, let $\vert \textbf{k}\vert =\left(k_{1}^{2}+\cdots+ k_{d}^{2}\right)^{1/2}$ and $\vert \textbf{k}\vert_{*} =\max_{1\leq j\leq d}  k_{j}$. Given $l,N\in\N_{0}$ we define $A_{l}=\left\{\textup{\textbf{k}}\in\N^{d}_{0}: \abs{\textup{\textbf{k}}}\leq l \right\}$, $A^{*}_{l}=\left\{\textup{\textbf{k}}\in\N^{d}_{0}: \abs{\textup{\textbf{k}}}_{*}\leq l \right\}$, $A_{-1}=A^{*}_{-1}=\emptyset$. Let $\left(\Omega_{1}, \mathcal{A}_{1}, \mu_{1}\right), \ldots, \left(\Omega_{d}, \mathcal{A}_{d}, \mu_{d}\right)$ be  probability spaces, $\Omega= \Omega_{1}\times \cdots \times \Omega_{d}$ and let $\mathcal{A}=\mathcal{A}_{1}\times \cdots\times \mathcal{A}_{d}$ and $\mu=\mu_{1}\times \cdots \times \mu_{d}$ be the product $\sigma$-algebra and the product measure, respectively.  Suppose $\{\phi_{\textup{\textbf{m}}}\}_{\textup{\textbf{m}}\in \N^{d}_{0}}$ is a complete orthonormal set of $L^{2}(\Omega, \mathcal{A}, \mu)$, of real-valued functions, uniformly bounded in $L^{\infty}(\Omega, \mathcal{A}, \mu)$. Let $\mathcal{H}_{l}=\left[\phi_{\textup{\textbf{k}}}: \textup{\textbf{k}}\in A_{l}\setminus A_{l-1}\right ]$ be the linear space generated by the functions $\phi_{\textup{\textbf{k}}}$ with $\textup{\textbf{k}} \in A_{l}\setminus A_{l-1}$ and let $d_{l}=\dim \mathcal{H}_{l}$, $\mathcal{T}_{N}=\bigoplus_{l=0}^{N}\mathcal{H}_{l}$, $\mathcal{H}=\bigcup_{N=1}^{\infty}\mathcal{T}_{N}$. Analogously we define $\mathcal{H}_{l}^{*}$, $d_{l}^{*}$, $\mathcal{T}_{N}^{*}$ and $\mathcal{H}^{*}=\mathcal{H}$. Given a real function $\lambda$ defined on the interval $[0,\infty)$ we consider the sequences $\Lambda=\{\lambda_{\textup{\textbf{k}}}\}_{\textup{\textbf{k}}\in\N^{d}_{0}}$ and $\Lambda^{*}=\{\lambda^{*}_{\textup{\textbf{k}}}\}_{\textup{\textbf{k}}\in \N_{0}^{d}}$ where $\lambda_{\textup{\textbf{k}}}=\lambda(\abs{\textup{\textbf{k}}})$, $\lambda_{\textup{\textbf{k}}}^{*}=\lambda(\abs{\textup{\textbf{k}}}_{*})$ and the linear operators associated with these sequences $\Lambda, \Lambda^{*}: \mathcal{H}\to \mathcal{H}$ defined for $f=\sum_{\textup{\textbf{k}}\in\N^{d}_{0}}a_{\textup{\textbf{k}}}\phi_{\textup{\textbf{k}}}\in \mathcal{H}$ by $\Lambda(f)=\sum_{\textup{\textbf{k}}\in\N^{d}_{0}} \lambda_{\textup{\textbf{k}}}a_{\textup{\textbf{k}}}\phi_{\textup{\textbf{k}}}$, $\Lambda^{*}(f)=\sum_{\textup{\textbf{k}}\in\N^{d}_{0}} \lambda^{*}_{\textup{\textbf{k}}}a_{\textup{\textbf{k}}}\phi_{\textup{\textbf{k}}}$. If $\Lambda$ and $\Lambda^{*}$ are bounded from $L^{p}(\Omega)$ to $L^{q}(\Omega)$ we also denote the extensions on $L^{p}(\Omega)$ by $\Lambda$ and $\Lambda^{*}$.

In the present paper, we study estimates for entropy numbers of multiplier operators of type $\Lambda$ and $\Lambda^{*}$.  The multiplier operators of types $\Lambda$ and $\Lambda_{*}$ were studied in \cite{Sergio, Stabile}.
In particular, we consider orthonormal systems $\{\phi_{\textup{\textbf{m}}}\}_{\textup{\textbf{m}}\in \N^{d}_{0}}$ of the type 
$\phi_{\textup{\textbf{m}}}(\textup{\textbf{x}})=\phi_{m_{1}}^{(1)}(x_{1})\cdots \phi_{m_{d}}^{(d)}(x_{d})$,
where each $\left\{\phi_{l}^{(j)}\right\}_{l\in \N_{0}}$ can be a different orthonormal system, for example, it can be a Vilenkin system, a Walsh system on a real sphere or the trigonometric system on the circle (see Remark \ref{examples}).


Let $\Lambda^{(1)}=\{\lambda_{\textbf{k}}\}_{\textbf{k}\in\N^{d}_{0}}$,  $\Lambda^{(1)}_{*}=\left\{\lambda^{*}_{\textbf{k}}\right\}_{\textbf{k}\in\N^{d}_{0}}$, where $\lambda(t)=t^{-\gamma}(\log_{2}t)^{-\xi}$ for $t>1$ and $\lambda(t)=0$ for $0\leq t\leq 1$, with $\gamma>0$, $\xi\geq 0$, and let $\Lambda^{(2)}=\left\{\lambda_{\textbf{k}}\right\}_{\textbf{k}\in\N^{d}_{0}}$,  $\Lambda^{(2)}_{*}=\left\{\lambda^{*}_{\textbf{k}}\right\}_{\textbf{k}\in\N^{d}_{0}}$, where $\lambda(t)=e^{-\gamma t^{r}}$, with $\gamma, r>0$. Let $U_{p}$ denote the closed unit ball in $L^{p}(\Omega)$.  
If $\{\phi_{\textup{\textbf{m}}}\}$ is the trigonometric system on the torus $\mathbb{T}^d$,
$\Lambda^{(1)}U_{p}$ and $\Lambda^{(1)}_{*}U_{p}$ are sets of finitely differentiable functions, in particular, are Sobolev-type classes if $\xi=0$, and $\Lambda^{(2)}U_{p}$ and $\Lambda^{(2)}_{*}U_{p}$ are sets of infinitely differentiable functions if $0<r<1$ and of analytic functions if $r\geq 1$ (see \cite{Stabile}).
If $\{\phi_{\textup{\textbf{m}}}\}$ is the multiple Walsh system on $[0,1]^d$,
$\Lambda^{(1)}U_{p}$ and $\Lambda^{(1)}_{*}U_{p}$ are sets of finitely differentiable functions and $\Lambda^{(2)}U_{p}$ and $\Lambda^{(2)}_{*}U_{p}$ are sets of infinitely differentiable functions, in the dyadic sense (see \cite{Sergio}).

For ease of notation we will write $a_{n}\gg b_{n}$ for two sequences if $a_{n}\geq C b_{n}$ for $n\in \N$, $a_{n}\ll b_{n}$ if $a_{n}\leq C b_{n}$ for $n\in \N$, and $a_{n}\asymp b_{n}$ if $a_{n}\ll b_{n}$ and $a_{n}\gg b_{n}$. For $a\in\N$ and $x\in\R$ such that $a\leq x<a+1$, we denote $[x]=a$.

The following estimates are the results of our applications and they are proved in Section 5.

\begin{theorem}\label{aplicacionesentropiafinitas1}
Consider $1\leq p, q\leq \infty$. 
If $\gamma> d/2$ and $\xi\geq 0$, then for all $n\in \N$,
	\begin{equation}\label{aplicacionfinitaentropia2}
	e_{n}\left(\Lambda^{(1)}U_{p},L^{q}\right)\ll n^{-\gamma/d}(\log_{2}n)^{-\xi}\left\{
	\begin{array}{ll}
q^{1/2},&\textbf{$2\leq p\leq \infty$, $q<\infty$},\\
	(\log_{2}n)^{1/2},&\textbf{$2\leq p\leq \infty$, $q=\infty$},
	\end{array}
	\right.
	\end{equation}
	and
	\begin{equation}\label{aplicacionfinitaentropia3}
	e_{n}\left(\Lambda^{(1)}U_{p},L^{q}\right)\gg n^{-\gamma/d}(\log_{2}n)^{-\xi}\left\{
	\begin{array}{ll}
	1,&\textbf{$ p< \infty$, $q>1$,}\\
	(\log_{2}n)^{-1/2},&\textbf{$p< \infty$, $q=1$},\\
	(\log_{2}n)^{-1/2},&\textbf{$ p= \infty$, $q>1$},\\
	(\log_{2}n)^{-1},&\textbf{$p= \infty$, $q=1$}.
	\end{array}
	\right.
	\end{equation}
	If $\gamma> d$ and $\xi\geq 0$, then for all $n\in \N$
	\begin{equation}\label{applicationfinite4}
	e_{n}\left(\Lambda^{(1)}U_{p},L^{q}\right)\ll n^{-\gamma/d}(\log_{2}n)^{-\xi}\left\{
	\begin{array}{ll}
q^{1/2}(p-1)^{-1/2},&\textbf{$1< p\leq 2$, $2\leq q<\infty$},\\
	q^{1/2}(\log_{2}n)^{1/2},&\textbf{$p=1$, $2\leq q<\infty$},\\
	(p-1)^{-1/2}(\log_{2}n)^{1/2},&\textbf{$1<p\leq 2$, $q=\infty$},\\
	\log_{2}n,&\textbf{$p=1$, $q=\infty$}.
	\end{array}
	\right.
	\end{equation}
	\end{theorem}

\begin{theorem}\label{teoremainfinitasentropia}
Consider $1\leq p, q\leq \infty$.
 Let $\gamma$, $r\in \R$ and suppose that $0<r\leq 1$ if $d\geq 4$, $0<r\leq 76/208$ if $d=2$ and $0<r\leq 11/16$ if $d=3$. Then for all $n\in\N$, we have that
	\begin{equation}\label{ecuacioninfinitasentropia2}
	e_{n}\left(\Lambda^{(2)}U_{p},L^{q}\right)\gg e^{-\mathcal{C}n^{r/(d+r)}}\left\{
	\begin{array}{ll}
	1,&\textbf{$ p<\infty$, $ q>1$},\\
	(\log_{2}n)^{-1/2},&\textbf{$ p< \infty$, $q=1$},\\
	(\log_{2}n)^{-1/2},&  p=\infty,q>1,\\
	(\log_{2} n)^{-1},& \textbf{$ p=\infty$, $q=1$},
	\end{array}
	\right.
	\end{equation}
	and
	\begin{equation}\label{ecuacioninfentropia}
	e_{n}\left(\Lambda^{(2)}U_{p},L^{q}\right)\ll e^{-\mathcal{C}n^{r/(d+r)}}\left\{
	\begin{array}{ll}
	q^{1/2},&\textbf{$ 2\leq p\leq\infty$, $q<\infty$},\\
	(\log_{2}n)^{1/2},&\textbf{$ 2\leq p\leq\infty$, $ q=\infty$},
	\end{array}
	\right.
	\end{equation}
	where
	\[
	\mathcal{C}=\gamma^{d/(d+r)}\left( \frac{(d+r)2^{d-1}d\Gamma(d/2)(\ln 2)}{r\pi^{d/2}}\right)^{r/(d+r)}.
	\]
\end{theorem}

The results of Theorem \ref{aplicacionesentropiafinitas1} also hold for the operator $\Lambda^{(1)}_{*}$ and the results of  Theorem \ref{teoremainfinitasentropia} hold for $\Lambda^{(2)}_{*}$ for all $0<r\leq 1$ and any dimension $d$, if we change the constant $\mathcal{C}$ by the constant $\mathcal{C}_{*}$, where
\[
\mathcal{C}_{*}=\gamma^{d/(d+r)}\biggl(\frac{(d+r)(\ln 2)}{r} \biggl)^{r/(d+r)}.
\]

The proofs in this paper are given only for the operators of type $\Lambda$. In Remark \ref{nuevaobser} we explain as the results can be proved for the operators of type $\Lambda_{*}$.

In this paper, the trigonometric system considered on the $d$-dimensional torus $\mathbb{T}^d$ is constructed by making the product of the usual real trigonometric system on the unit circle (see Proposition \ref{proposi1} and Remark \ref{examples}).
In  \cite{Stabile}, estimates for entropy numbers of the operator  
$\Lambda^{(2)}$ were studied considering another trigonometric system on the torus. The upper and lower estimates were obtained for all real number $r$, $0< r \leq 1$ and for any dimension $d$. We note that the estimates in \cite{Stabile} when $d=2$ are true only if
$0<r\leq 76/208$ and for $d=3$ only if $0<r\leq 11/16$.

The estimates of Theorems \ref{aplicacionesentropiafinitas1} and \ref{teoremainfinitasentropia}
are sharp in various important situations. For $2\leq p,q <\infty$,  we have that 
$$
e_{n}\left(\Lambda^{(1)}U_{p},L^{q}\right)\asymp e_{n}\left(\Lambda^{(1)}_{*}U_{p},L^{q}\right)\asymp n^{-\gamma/d}(\log_{2} n)^{-\xi}, 
\quad \gamma>d/2, \; \xi\geq 0;
$$ 
$$
e_{n}\left(\Lambda^{(2)}_{*}U_{p},L^{q}\right)\asymp e^{-\mathcal{C}_{*}n^{r/(d+r)}}, \quad 0<r\leq 1, \, \gamma >0;
$$
and
$$
e_{n}\left(\Lambda^{(2)}U_{p},L^{q}\right)\asymp e^{-\mathcal{C}n^{r/(d+r)}}, 
$$
for $0<r\leq 1$ if $d\geq 4$, $0<r\leq 76/208$ if $d=2$, $0<r\leq 11/16$ if $d=3$ and for $\gamma >0$.


The organization of the paper is as follows. In Section 2 we introduce real-valued Vilenkin systems. In Section 3 we recall some relevant facts and establish a few auxiliary results for later use.
In Section 4 we prove two theorems where upper and lower bounds are established for entropy numbers of general multiplier operators. 
The paper ends with the proofs of Theorems \ref{aplicacionesentropiafinitas1} and \ref{teoremainfinitasentropia}  and two remarks in Section 5.

\section{Real-valued Vilenkin system}

Let $(s_{j})_{j\in\N_{0}}$ be a sequence of integers with $s_{j}\geq2$
and let $P_{0}=1$, $P_{k}=\prod_{j=0}^{k-1}s_{j}$ for $k\geq1$.
For an integer $s\geq2$, consider the group $\Z_{s}=\Z/(s\Z)$ with
the discrete topology and the group $\G=\prod_{k=0}^{\infty}\Z_{s_{k}}$
with the product topology. Denote by $\nu$ the normalized Haar measure   
on $\G$. We can identify $\G$ with the unit interval $[0,1]$, associating
each $(x_{k})\in\G$ with the point $|x|=\sum_{k=0}^{\infty}x_{k}/P_{k+1}$.
If we disregard the countable subset
\[
\G_{0}=\left\{ (x_{j}):\,\exists k\in\N_{0}:x_{j}=0\text{, }\forall j\geq k\right\} \subset\G\text{,}
\]
this mapping is one-to-one and onto $(0,1]$. If $A$ is a measurable
subset of $\G$ then the measure $\nu(A)$ is equal to the Lebesgue
measure of the set $|A|=\left\{ |x|:x\in A\right\} $. For $x=(x_{k})\in\G$
we define ${\rho}_{k}(x)=\exp(2\pi ix_{k}/s_{k})$, $k\in\N_{0}$.
For $n\in\N_{0}$, let $(n_{k})_{k\in\N}$, $n_{k}\in\{0,1,\ldots,s_{k}-1\}$,
be the unique sequence such that $n=\sum_{k=0}^{\infty}n_{k}P_{k}$.
We define the (complex-valued) Vilenkin function ${\psi}_{n}$ \cite{Vilenkin,Watari}
by
\[
{\psi}_{n}(x)=\prod_{k=0}^{\infty}\rho_{k}^{n_{k}}(x).
\]
If $s_{k}=2$ for all $k\in\N_{0}$, the Vilenkin functions are the
Walsh functions. The Walsh system has been extensively used in data
transmission, filtering, image enhancement, signal analysis and pattern
recognition \cite{Bar,Schipp}.

We denote by $L^{p}(\G)$, $1\leq p\leq\infty$, the usual vector space
of complex measurable functions $f$ on $\G$.
The Vilenkin-Fourier system $(\psi_{n})_{n\in\N}$ is a complete orthonormal
set in $L^{2}(\G)$ \cite{Vilenkin}. For $f\in L^{1}(\G)$,
we define the $n^{\text{th}}$ partial sum $S_{n}\left(f\right)$
of the Vilenkin-Fourier series of $f$ by 
\[
S_{n}\left(f\right)=\sum_{k=0}^{n-1}\hat{f}\left(k\right)\psi_{k}\text{,}\qquad\hat{f}\left(k\right)=\int_{\G}f\left(t\right)\overline{\psi_{k}}\left(t\right)d\nu(t).
\]
It is known \cite{Young} that, if $1<p<\infty$ and $f\in L^{p}(\G)$,
then $S_{n}\left(f\right)$ converges to $f$ in $L^{p}(\G)$.

The Fourier system $\left\{ \cos(nx):n\in\N_{0}\right\} \cup\left\{ \sin(nx):n\in\N\right\} $
can be obtained from the system \linebreak $\left\{ \exp\left(inx\right):n\in\Z\right\} $
by taking the real and imaginary part of $\exp\left(inx\right)$ for
$n\in\N_{0}$. We proceed similarly to obtain a complete orthonormal
system of real-valued functions from the complex-valued Vilenkin functions,
but some care is needed, as can be seen in the following examples.
\begin{example}
If $s_{k}=2$ for all $k\in\N_{0}$, we have ${\rho}_{k}(x)=\left(-1\right)^{x_{k}}$
for all $k$ and ${\psi}_{n}$ is real-valued for all $n\in\N_{0}$.
\end{example}

\begin{example}
\label{exa:s_k-odd}Suppose $s_{k}$ is odd for all $k\in\N_{0}$.
Then ${\rho}_{k}(x)=\omega_{s_{k}}^{x_{k}}$ for all $k$, where $\omega_{s}=\exp(2\pi i/s)$.
Note that $\omega_{s}^{u}=\exp\left(\pi i(2u/s)\right)$ is not a
real number unless $2u/s\in\Z$. Now, let $n\in\N$ and fix $j\in\N_{0}$
such that $n_{j}\neq0$. Let $x\in\G$ with $x_{j}=1$ and $x_{k}=0$
for $k\neq j$. Then
\[
{\psi}_{n}(x)=\prod_{k=0}^{\infty}\rho_{k}^{n_{k}}(x)=\prod_{k=0}^{\infty}\omega_{s_{k}}^{n_{k}x_{k}}=\omega_{s_{j}}^{n_{j}}
\]
is not a real number because $2n_{j}/s_{j}\notin\Z$. Therefore, ${\psi}_{n}$
is real-valued if and only if $n=0$.
\end{example}

\begin{example}
Suppose $s_{k}=4$ for all $k\in\N_{0}$. Then ${\rho}_{k}(x)=i^{x_{k}}$
for all $k$. Let $n\in\N$. If $n_{k}=2u_{k}$ for all $k$, then
\[
{\psi}_{n}(x)=\prod_{k=0}^{\infty}\rho_{k}^{n_{k}}(x)=\prod_{k=0}^{\infty}i^{2u_{k}x_{k}}=\prod_{k=0}^{\infty}\left(-1\right)^{x_{k}}
\]
is a real number for all $x\in\G$. Otherwise, fix $j\in\N_{0}$ such
that $n_{j}$ is odd and let $x\in\G$ with $x_{j}=1$ and \linebreak $x_{k}=0$
for $k\neq j$. Then

\[
{\psi}_{n}(x)=\prod_{k=0}^{\infty}\rho_{k}^{n_{k}}(x)=\prod_{k=0}^{\infty}i^{n_{k}x_{k}}=i^{n_{j}}
\]
is not a real number. Therefore, ${\psi}_{n}$ is real-valued if and
only if $n_{k}$ is even for all $k\in\N_{0}$.
\end{example}

\begin{example}
Let $s_{2k}=2$ and $s_{2k+1}=3$ for all $k\in\N_{0}$. In this case,
${\rho}_{2k}(x)=\left(-1\right)^{x_{2k}}$ and ${\rho}_{2k+1}(x)=\omega_{3}^{x_{2k+1}}$
where $\omega_{3}=\exp(2\pi i/3)$. Let $n\in\N_{0}$. If $n_{2k+1}=0$
for all $k$, we have
\[
{\psi}_{n}(x)=\prod_{k=0}^{\infty}\rho_{k}^{n_{k}}(x)=\prod_{k=0}^{\infty}\rho_{2k}^{n_{2k}}(x)=\prod_{k=0}^{\infty}\left(-1\right)^{n_{2k}x_{2k}}
\]
is real-valued. Otherwise, fix $j\in\N_{0}$ such that $n_{2j+1}\neq0$
and let $x\in\G$ with $x_{2j+1}=1$ and $x_{k}=0$ for $k\neq2j+1$.
Then ${\psi}_{n}(x)=\omega_{3}^{n_{2k+1}}$ is not a real number (see
Example \ref{exa:s_k-odd}). Therefore, ${\psi}_{n}$ is real-valued
if and only if $n_{2k+1}=0$ for all $k\in\N_{0}$.
\end{example}

For $n,m\in\N_{0}$, let
\begin{equation*}
n\oplus m =\sum_{k=0}^{\infty}\left(\left(n_{k}+m_{k}\right)\mod s_{k}\right)P_{k}, \quad
n\ominus m =\sum_{k=0}^{\infty}\left(\left(n_{k}-m_{k}\right)\mod s_{k}\right)P_{k}
\end{equation*}
and $\ominus n=0\ominus n$. If $u=n\oplus m$, then $u_{k}=n_{k}+m_{k}-\alpha_{k}s_{k}$
with $\alpha_{k}\in\N_{0}$ and $\psi_{n\oplus m}\left(x\right) = \psi_{n}\left(x\right)\psi_{m}\left(x\right)$.
In particular, $\psi_{n}\left(x\right)\psi_{\ominus n}\left(x\right)=1$
and, since $\left|\psi_{n}\left(x\right)\right|=1$, we have $\psi_{\ominus n}\left(x\right)=\overline{\psi_{n}\left(x\right)}$.
Therefore, $\psi_{n}$ is real-valued if and only if $n=\ominus n$.
We define
$$
\mathbb{K} =\left\{ n\in\N_{0}:\,n<\ominus n\right\} , \quad
\mathbb{L} =\left\{ n\in\N_{0}:\,n>\ominus n\right\} , \quad
\mathbb{M} =\left\{ n\in\N_{0}:\,n=\ominus n\right\} .
$$
For $n\in\mathbb{K}\cup\mathbb{L}$, let $X_{n}=\sqrt{2}\Re\left(\psi_{n}\right)$ and $Y_{n}=\sqrt{2}\Im\left(\psi_{n}\right)$, 
where $\Re\left(f\right)$ denotes the real part and $\Im\left(f\right)$ the imaginary part of a complex function $f$.
For $n\in\mathbb{M}$, let $X_{n}=\psi_{n}$ and $Y_{n}\equiv0$. 
\begin{proposition}
The set
\[
\Psi=\left\{ X_{n}:\,n\in\mathbb{K}\cup\mathbb{M}\right\} \cup\left\{ Y_{n}:\,n\in\mathbb{K}\right\} 
\]
 is a complete orthonormal system of $L^{2}\left(\G\right)$.
\end{proposition}

\begin{proof}
If $n\in\mathbb{M}$, then $\left\Vert X_{n}\right\Vert _{2}=\left\Vert \psi_{n}\right\Vert _{2}=1$.
If $n\in\mathbb{K}$, then $\psi_{n}=2^{-1/2}\left(X_{n}+iY_{n}\right)$.
Since $\left\Vert \psi_{n}\right\Vert _{2}=1$,
\[
2=2\int_{\G}\left|\psi_{n}\right|^{2}d\nu=\int_{\G}\left|X_{n}+iY_{n}\right|^{2}d\nu=\int_{\G}X_{n}^{2}d\nu+\int_{\G}Y_{n}^{2}d\nu,
\]
that is, $\left\Vert X_{n}\right\Vert _{2}^{2}+\left\Vert Y_{n}\right\Vert _{2}^{2}=2$.
Also, since $n\neq\ominus n$ and $\psi_{n},\psi_{\ominus n}$ are
orthogonal
\begin{align*}
0 & =2\int_{\G}\psi_{n}\psi_{\ominus n}d\nu=\int_{\G}\left(X_{n}+iY_{n}\right)\left(X_{n}-iY_{n}\right)\,d\nu\\
 & =\int_{\G}\left(X_{n}^{2}-Y_{n}^{2}\right)d\nu+i\int_{\G}\left(X_{n}Y_{n}+Y_{n}X_{n}\right)d\nu.
\end{align*}
From the real part, we obtain $\left\Vert X_{n}\right\Vert _{2}^{2}-\left\Vert Y_{n}\right\Vert _{2}^{2}=0$
and therefore$\left\Vert X_{n}\right\Vert _{2}=\left\Vert Y_{n}\right\Vert _{2}=1$.
From the imaginary part, it follows that $\int_{\G}X_{n}Y_{n}d\nu=0$.
If $n\in\mathbb{K}$ and $m\in\mathbb{M}$, since $\psi_{n},\psi_{m}$
are orthogonal,
\[
0=\sqrt{2}\int_{\G}\psi_{n}\psi_{m}d\nu=\int_{\G}\left(X_{n}+iY_{n}\right)X_{m}\,d\nu.
\]
Therefore, $\int_{\G}X_{n}X_{m}d\nu=\int_{\G}Y_{n}X_{m}d\nu=0$. Finally,
if $n,m\in\mathbb{K}$ with $n\neq m$, we have
\begin{align*}
0 & =2\int_{\G}\psi_{n}\psi_{m}d\nu=\int_{\G}\left(X_{n}+iY_{n}\right)\left(X_{m}+iY_{m}\right)\,d\nu\\
 & =\int_{\G}\left(X_{n}X_{m}-Y_{n}Y_{m}\right)d\nu+i\int_{\G}\left(X_{n}Y_{m}+Y_{n}X_{m}\right)d\nu.
\end{align*}
We also have $n\neq\ominus m$ (otherwise $m<\ominus m=n<\ominus n=m$,
a contradiction). Thus, $\psi_{n}$ and $\psi_{\ominus m}$ are orthogonal
and
\begin{align*}
0 & =2\int_{\G}\psi_{n}\psi_{\ominus m}d\nu=\int_{\G}\left(X_{n}+iY_{n}\right)\left(X_{m}-iY_{m}\right)\,d\nu\\
 & =\int_{\G}\left(X_{n}X_{m}+Y_{n}Y_{m}\right)d\nu+i\int_{\G}\left(-X_{n}Y_{m}+Y_{n}X_{m}\right)d\nu.
\end{align*}
We obtain $\int_{\G}X_{n}X_{m}d\nu=\int_{\G}Y_{n}X_{m}d\nu=0$ from
the sum of the equations and $\int_{\G}Y_{n}Y_{m}d\nu=\int_{\G}X_{n}Y_{m}d\nu=0$,
from the subtraction. Therefore, we conclude that the set $\Psi$
is an orthonormal system.

To prove that $\Psi$ is complete, let $f\in L^{2}\left(\G\right)$
and supose that $\int_{\G}fX_{n}\,d\nu=0$ for all $n\in\mathbb{K}\cup\mathbb{M}$
and $\int_{\G}fY_{n}\,d\nu=0$ for all $n\in\mathbb{K}$. We have
$\psi_{n}=2^{-1/2}\left(X_{n}+iY_{n}\right)$ for $n\in\mathbb{K}$
and $\psi_{n}=X_{n}$ for $n\in\mathbb{M}$. For $n\in\mathbb{L}$,
we have $\ominus n\in\mathbb{K}$ and $\psi_{n}=2^{-1/2}\left(X_{\ominus n}-iY_{\ominus n}\right)$.
Therefore, $\int_{\G}f\psi_{n}d\nu=0$ for all $n\in\N_{0}$. But
$\left(\psi_{n}\right)_{n\in\N_{0}}$ is a complete orthonormal system
of $L^{2}\left(\G\right)$ and therefore $f=0$ a.e.
\end{proof}
The following proposition shows a simple way of determining whether
$n$ belongs to $\mathbb{K}$, $\mathbb{L}$ or $\mathbb{M}$.
\begin{proposition}
\label{prop:sep-k-l-m}Let $n\in\N$ and $r\in\N_{0}$ such that $P_{r}\leq n<P_{r+1}$
(in particular, $n=\sum_{k=0}^{r}n_{k}P_{k}$ and $n_{r}>0$). Let
$n'=\sum_{k=0}^{r-1}n_{k}P_{k}=n-n_{r}P_{r}$. We have:
\begin{enumerate}
\item If $n_{r}<s_{r}/2$, then $n\in\mathbb{K}$.
\item If $n_{r}>s_{r}/2$, then $n\in\mathbb{L}$.
\item If $n_{r}=s_{r}/2$ (in particular, $s_{r}$ is even) then $n\in A\iff n'\in A$
for $A=\mathbb{K},\mathbb{L},\mathbb{M}$.
\end{enumerate}
\end{proposition}

\begin{proof}
We have $\ominus n=\sum_{k=0}^{r}\tilde{n}_{k}P_{k}$ where $\tilde{n}_{k}=\left(s_{k}-n_{k}\right)\mod s_{k}$.
If $n_{r}<s_{r}/2$ then $\tilde{n}_{r}=s_{r}-n_{r}>s_{r}/2>n_{r}$
and
\[
n=n'+n_{r}P_{r}<P_{r}+n_{r}P_{r}\leq\tilde{n}_{r}P_{r}\leq\ominus n.
\]
Thus $n\in\mathbb{K}$. If $n_{r}>s_{r}/2$, then $\tilde{n}_{r}<s_{r}/2$
and, by the previous case, $\tilde{n}\in\mathbb{K}$, that is, $n\in\mathbb{L}$.
Now, suppose $n_{r}=s_{r}/2$. In this case, $\tilde{n}_{r}=s_{r}-n_{r}=n_{r}$
and $\ominus n'=\ominus n-\tilde{n}_{r}P_{r}=\ominus n-n_{r}P_{r}$.
Since $n=n'+n_{r}P_{r}$ and $\ominus n=\ominus n'+n_{r}P_{r}$,
\[
n\in\mathbb{K}\iff n<\ominus n\iff n'<\ominus n'\iff n'\in\mathbb{K}.
\]
Symmilarly, $n\in\mathbb{L}\iff n'\in\mathbb{L}$ and $n\in\mathbb{M}\iff n'\in\mathbb{M}$.
\end{proof}
Now we consider the problem of ordering $\Psi$. One simple ordering
is defined by
\[
Z_{n}=\begin{cases}
X_{n}, & n\in\mathbb{K}\cup\mathbb{M},\\
Y_{\ominus n}, & n\in\mathbb{L}.
\end{cases}
\]
Since $n\in\mathbb{L}\mapsto\ominus n\in\mathbb{K}$ and $\left\{ Y_{n}: \,n\in\mathbb{K}\right\} =\left\{ Y_{\ominus n}: \,n\in\mathbb{L}\right\} $,
we have $\left\{ Z_{n}: \,n\in\N_{0}\right\} =\Psi$. We can also consider
the system
\[
\tilde{Z}_{n}=\begin{cases}
X_{n}, & n\in\mathbb{K}\cup\mathbb{M},\\
Y_{n}, & n\in\mathbb{L}.
\end{cases}
\]
Note that, for all $n\in\N_{0}$, we have $\psi_{\ominus n}=\overline{\psi_{n}}$.
Thus, $X_{\ominus n}=X_{n}$ and $Y_{\ominus n}=-Y_{n}$, that is,
$\tilde{Z}_{n}=Z_{\ominus n}$ and $\left\{\tilde{Z}_{n} : \,n\in\N_{0}\right\}$ is an ordering
of 
\[
\tilde{\Psi}=\left\{ X_{n}:\,n\in\mathbb{K}\cup\mathbb{M}\right\} \cup\left\{ -Y_{n}:\,n\in\mathbb{K}\right\} 
\]
which is also a complete orthonormal system.

We have that $\left\Vert Z_{n}\right\Vert _{\infty}=\left\Vert \tilde{Z}_{n}\right\Vert _{\infty} \leq \sqrt{2}$
for all $n \in \N_{0}$. From Proposition \ref{prop:sep-k-l-m}, if $P_{r}\leq n<P_{r+1}$, then we also have $P_{r}\leq \ominus n<P_{r+1}$.
Fixed a sequence $\left(s_j \right)$ and given $n \in \N_{0}$, Proposition \ref{prop:sep-k-l-m}  will facilitate the identification of the functions $Z_n$ and $\tilde{Z}_{n}$, especially in concrete cases as in the previous examples.

For the results of this study, any ordering of $\Psi$ or $\tilde{\Psi}$
is sufficient, since reordering preserves orthonormality and completeness.

\begin{example}
Suppose $s_{k}=3$ for all $k\in\N_{0}$. The following table shows
$Z_{n}$ and $\tilde{Z}_{n}$ for $n=0,\dots,27$. In the table, $A_{n}=\mathbb{K}\iff n\in\mathbb{K}$,
$A_{n}=\mathbb{L}\iff n\in\mathbb{L}$ and $A_{n}=\mathbb{M}\iff n\in\mathbb{M}$.
\[
\begin{array}{cccccccc|cccccccc}
n & n_{2} & n_{1} & n_{0} & \ominus n & A_{n} & Z_{n} & \tilde{Z}_{n} & n & n_{2} & n_{1} & n_{0} & \ominus n & A_{n} & Z_{n} & \tilde{Z}_{n}\\
0 & 0 & 0 & 0 & 0 & \mathbb{M} & X_{0} & X_{0} & 14 & 1 & 1 & 2 & 25 & \mathbb{K} & X_{14} & X_{14}\\
1 & 0 & 0 & 1 & 2 & \mathbb{K} & X_{1} & X_{1} & 15 & 1 & 2 & 0 & 21 & \mathbb{K} & X_{15} & X_{15}\\
2 & 0 & 0 & 2 & 1 & \mathbb{L} & Y_{1} & Y_{2} & 16 & 1 & 2 & 1 & 23 & \mathbb{K} & X_{16} & X_{16}\\
3 & 0 & 1 & 0 & 6 & \mathbb{K} & X_{3} & X_{3} & 17 & 1 & 2 & 2 & 22 & \mathbb{K} & X_{17} & X_{17}\\
4 & 0 & 1 & 1 & 8 & \mathbb{K} & X_{4} & X_{4} & 18 & 2 & 0 & 0 & 9 & \mathbb{L} & Y_{9} & Y_{18}\\
5 & 0 & 1 & 2 & 7 & \mathbb{K} & X_{5} & X_{5} & 19 & 2 & 0 & 1 & 11 & \mathbb{L} & Y_{11} & Y_{19}\\
6 & 0 & 2 & 0 & 3 & \mathbb{L} & Y_{3} & Y_{6} & 20 & 2 & 0 & 2 & 10 & \mathbb{L} & Y_{10} & Y_{20}\\
7 & 0 & 2 & 1 & 5 & \mathbb{L} & Y_{5} & Y_{7} & 21 & 2 & 1 & 0 & 15 & \mathbb{L} & Y_{15} & Y_{21}\\
8 & 0 & 2 & 2 & 4 & \mathbb{L} & Y_{4} & Y_{8} & 22 & 2 & 1 & 1 & 17 & \mathbb{L} & Y_{17} & Y_{22}\\
9 & 1 & 0 & 0 & 18 & \mathbb{K} & X_{9} & X_{9} & 23 & 2 & 1 & 2 & 26 & \mathbb{L} & Y_{16} & Y_{23}\\
10 & 1 & 0 & 1 & 20 & \mathbb{K} & X_{10} & X_{10} & 24 & 2 & 2 & 0 & 23 & \mathbb{L} & Y_{12} & Y_{24}\\
11 & 1 & 0 & 2 & 19 & \mathbb{K} & X_{11} & X_{11} & 25 & 2 & 2 & 1 & 14 & \mathbb{L} & Y_{14} & Y_{25}\\
12 & 1 & 1 & 0 & 24 & \mathbb{K} & X_{12} & X_{12} & 26 & 2 & 2 & 2 & 13 & \mathbb{L} & Y_{13} & Y_{26}\\
13 & 1 & 1 & 1 & 26 & \mathbb{K} & X_{13} & X_{13} & 27 & 0 & 0 & 0 & 54 & \mathbb{K} & X_{27} & X_{27}
\end{array}
\]

\end{example}

\begin{example}
Suppose $s_{k}=4$ for all $k\in\N_{0}$. The following table shows
$Z_{n}$ and $\tilde{Z}_{n}$ for $n=0,\dots,15$.
\[
\begin{array}{ccccccc|ccccccc}
n & n_{1} & n_{0} & \ominus n & A_{n} & Z_{n} & \tilde{Z}_{n} & n & n_{1} & n_{0} & \ominus n & A_{n} & Z_{n} & \tilde{Z}_{n}\\
0 & 0 & 0 & 0 & \mathbb{M} & X_{0} & X_{0} & 8 & 2 & 0 & 8 & \mathbb{M} & X_{8} & X_{8}\\
1 & 0 & 1 & 3 & \mathbb{K} & X_{1} & X_{1} & 9 & 2 & 1 & 11 & \mathbb{K} & X_{9} & X_{9}\\
2 & 0 & 2 & 2 & \mathbb{M} & X_{2} & X_{2} & 10 & 2 & 2 & 10 & \mathbb{M} & X_{10} & X_{10}\\
3 & 0 & 3 & 1 & \mathbb{L} & Y_{1} & Y_{3} & 11 & 2 & 3 & 9 & \mathbb{L} & Y_{9} & Y_{11}\\
4 & 1 & 0 & 12 & \mathbb{K} & X_{4} & X_{4} & 12 & 3 & 0 & 4 & \mathbb{L} & Y_{4} & Y_{12}\\
5 & 1 & 1 & 15 & \mathbb{K} & X_{5} & X_{5} & 13 & 3 & 1 & 7 & \mathbb{L} & Y_{7} & Y_{13}\\
6 & 1 & 2 & 14 & \mathbb{K} & X_{6} & X_{6} & 14 & 3 & 2 & 6 & \mathbb{L} & Y_{6} & Y_{14}\\
7 & 1 & 3 & 13 & \mathbb{K} & X_{7} & X_{7} & 15 & 3 & 3 & 5 & \mathbb{L} & Y_{5} & Y_{15}
\end{array}
\]
The following table shows $Z_{n}$ and $\tilde{Z}_{n}$ for $n_{2}=2$,
that is, $n=32,\dots,47$ (note that the column $A_{n}$ in table
below is equal to the column $A_{n}$ in the table above - that is
a consequence of Proposition \ref{prop:sep-k-l-m}).
\[
\begin{array}{cccccccc|cccccccc}
n & n_{2} & n_{1} & n_{0} & \ominus n & A_{n} & Z_{n} & \tilde{Z}_{n} & n & n_{2} & n_{1} & n_{0} & \ominus n & A_{n} & Z_{n} & \tilde{Z}_{n}\\
32 & 2 & 0 & 0 & 32 & \mathbb{M} & X_{32} & X_{32} & 40 & 2 & 2 & 0 & 40 & \mathbb{M} & X_{40} & X_{40}\\
33 & 2 & 0 & 1 & 35 & \mathbb{K} & X_{33} & X_{33} & 41 & 2 & 2 & 1 & 43 & \mathbb{K} & X_{41} & X_{41}\\
34 & 2 & 0 & 2 & 34 & \mathbb{M} & X_{34} & X_{34} & 42 & 2 & 2 & 2 & 42 & \mathbb{M} & X_{42} & X_{42}\\
35 & 2 & 0 & 3 & 33 & \mathbb{L} & Y_{33} & Y_{35} & 43 & 2 & 2 & 3 & 41 & \mathbb{L} & Y_{41} & Y_{43}\\
36 & 2 & 1 & 0 & 44 & \mathbb{K} & X_{36} & X_{36} & 44 & 2 & 3 & 0 & 36 & \mathbb{L} & Y_{36} & Y_{44}\\
37 & 2 & 1 & 1 & 47 & \mathbb{K} & X_{37} & X_{37} & 45 & 2 & 3 & 1 & 39 & \mathbb{L} & Y_{39} & Y_{45}\\
38 & 2 & 1 & 2 & 46 & \mathbb{K} & X_{38} & X_{38} & 46 & 2 & 3 & 2 & 38 & \mathbb{L} & Y_{38} & Y_{46}\\
39 & 2 & 1 & 3 & 45 & \mathbb{K} & X_{39} & X_{39} & 47 & 2 & 3 & 3 & 37 & \mathbb{L} & Y_{37} & Y_{47}
\end{array}
\]
\end{example}

\section{Main definitions and some results}


We denoted by $L^{p}=L^{p}(\Omega)$, $1\leq p\leq \infty$, the vector space consisting of all measurable functions $f$ defined on $\Omega$ and with values in $\R$, satisfying 
\[
\norm{f}_{p}=\norm{f}_{L^{p}(\Omega)}=\left ( \int_{\Omega}\abs{f(\textbf{x})}^{p}d\mu(\textbf{x}) \right)^{1/p}<\infty, \hspace{3ex} 1\leq p <\infty,
\]
\[
\norm{f}_{\infty}=\norm{f}_{L^{\infty}(\Omega)}=\textup{ess}\sup_{\textbf{x}\in \Omega} \abs{f(\textbf{x})}< \infty.
\]
\begin{proposition}\label{proposi1}
For each $1\leq j\leq d$, let $\left\{\phi_{l}^{(j)}\right\}_{l\in \N_{0}}$ be a complete orthonormal set of $L^{2}\left(\Omega_{j},\mathcal{A}_{j},\mu_{j}\right)$, of real-valued or complex-valued functions. For $\textup{\textbf{m}}=(m_{1},\ldots,m_{d})\in \N^{d}_{0}$, consider the function $\phi_{\textup{\textbf{m}}}:\Omega\to \C$ given by
\[
\phi_{\textup{\textbf{m}}}(\textup{\textbf{x}})=\phi_{m_{1}}^{(1)}(x_{1})\cdots \phi_{m_{d}}^{(d)}(x_{d}), \hspace{2ex} \textup{\textbf{x}}=(x_{1},\ldots,x_{d})\in \Omega.
\]
Then $\left\{\phi_{\textup{\textbf{m}}}\right\}_{\textup{\textbf{m}}\in \N_{0}^{d}}$ is a complete orthonormal set of $L^{2}(\Omega,\mathcal{A}, \mu)$.
\end{proposition}
\begin{proof}
    It is immediate that the set $\{\phi_{\textup{\textbf{m}}}\}_{\textup{\textbf{m}}\in \N_{0}^{d}}$ is orthonormal. Consider $d=2$ and let $f\in L^{2}(\Omega, \mathcal{A}, \mu)$ such that
    \[
    \int_{\Omega} f(\textup{\textbf{x}})\overline{\phi_{\textup{\textbf{m}}}}(\textup{\textbf{x}})d\mu (\textup{\textbf{x}})=0
    \]
    for all $\textup{\textbf{m}}=(m_{1},m_{2})\in \N_{0}^{2}$. Then fixed $m_{2}\in\N_{0}$, for all $m_{1}\in \N_{0}$ we have
    \[
        0=\int_{\Omega_{1}}\left(\int_{\Omega_{2}} f(x_{1},x_{2})\overline{\phi_{m_{2}}^{(2)}}(x_{2})d\mu_{2} (x_{2}) \right)\overline{\phi_{m_{1}}^{(1)}}(x_{1})d\mu_{1} (x_{1}).
    \]
      Let 
    \[
    E_{m_{2}}=\left\{x_{1}\in \Omega_{1}: F_{m_2}\left(x_{1}\right) =\int_{\Omega_{2}}f(x_{1},x_{2})\overline{\phi_{m_{2}}^{(2)}}(x_{2})d\mu_{2} (x_{2}) \neq 0 \right\}.
    \]
    Since $\left\{\phi_{k}^{(1)}\right\}_{k\in\N_{0}}$ is a complete orthonormal set of $L^{2}\bigl(\Omega_{1},\mathcal{A}_{1},\mu_{1}\bigl)$, the function $F_{m_2}$ is null almost everywhere and thus $\mu_{1}\bigl(E_{m_{2}}\bigl)=0$. If $E=\cup_{m_{2}\in\N_{0}}E_{m_{2}}$ then $\mu_{1}(E)=0$ and $F_{m_2}\left(x_1\right)=0$ for all $x_{1}\in \Omega_{1}\setminus E$ and all $m_{2}\in \N_{0}$. Therefore, again by hypothesis, we have that $f(x_{1},\cdot)=0$ a.e. for all $x_{1}\in \Omega_{1} \setminus E$ and then
    \[
    \int_{\Omega}\abs{f(\textup{\textbf{x}})}d\mu(\textup{\textbf{x}})=\int_{\Omega_{1}\setminus E} \left(\int_{\Omega_{2}} \abs{f(x_{1},x_{2})}d\mu_{2}(x_{2}) \right)d\mu_{1}(x_{1})=0.
    \]
    Hence $f=0$ in $L^{2}(\Omega)$. The general case we prove proceeding by induction on $d$.
\end{proof}

From this point on, we will consider a complete orthonormal set of real-valued functions $\{\phi_{\textup{\textbf{m}}}\}_{\textup{\textbf{m}}\in\N_{0}^{d}}$ of \linebreak $L^{2}(\Omega, \mathcal{A},\mu)$ satisfying the following condition: there is a constant $K>0$ such that $\norm{\phi_{\textup{\textbf{m}}}}_{\infty}\leq K$ for all $\textup{\textbf{m}}\in \N_{0}^{d}$.

If for each $1\leq j\leq d$, $\left\{\phi_{l}^{(j)}\right\}_{l\in \N_{0}}$ is a complete orthonormal set of real-valued  functions of $L^{2}\left(\Omega_{j},\mathcal{A}_{j},\mu_{j}\right)$ and for each $1\leq j \leq d$, there is a constant $K_{j}>0$ such that $\norm{\phi_{l}^{(j)}}_{\infty}\leq K_{j}$ for all $l\in \N_{0}$, then the complete orthonormal set of $L^{2}(\Omega, \mathcal{A}, \mu)$ with the functions $\phi_{\textup{\textbf{m}}}$ as in the Proposition \ref{proposi1}, satisfies $\norm{\phi_{\textup{\textbf{m}}}}_{\infty}\leq K$ for \linebreak $K=\sup_{1\leq j\leq d}K_{j}$ and for all $\textup{\textbf{m}}\in\N^{d}_{0}$.

We write $U_{p}=\left\lbrace \varphi \in L^{p}: \norm{\varphi}_{p}\leq 1\right\rbrace$. For $f\in L^{1}(\Omega)$ the $d$-dimensional Fourier series of $f$ is given by 
\[
\sum_{\textbf{m}\in \N_{0}^{d}}\widehat{f}(\textbf{m})\phi_{\textbf{m}},\hspace{2ex} \widehat{f}(\textbf{m})=\int_{\Omega} f(\textbf{x})\phi_{\textbf{m}}(\textbf{x})d\mu (\textbf{x}).
\]
For $f\in L^{1}(\Omega)$ and $R>0$, we define the spherical partial sum of the Fourier series of the function $f$ by 
\[
S_{R}(f)=\sum_{{\shortstack{$\scriptstyle {\textbf{m}\in \N_{0}^{d}}$\\
			$\scriptstyle {\abs{\textbf{m}}\leq R}$}}}\widehat{f}(\textbf{m})\phi_{\textbf{m}}.
\]

 Let $A$ be a compact subset of a Banach space $X$ and let $\epsilon>0$. A finite subset $S=\{x_{1}, x_{2},\ldots,x_{m}\}$ of $X$ is called an $\epsilon$-net for $A$ in $X$ if, for each $x\in A$, there is at least one point $x_{k}\in S$ such that $\norm{x_{k}-x}\leq \epsilon$, that is, $A\subset \cup_{k=1}^{m}(x_{k}+\epsilon B_{X})$. The set $S=\left\{x_{1}, x_{2},\ldots,x_{m}\right\}$ is called an $\epsilon$-distinguishable subset of $A$ in $X$, if $S\subset A$ and $\norm{x_{i}-x_{j}}>\epsilon$ for all $1\leq i$, $j\leq m$, $i\neq j$. If every $\epsilon$-distinguishable subset of $A$ in $X$ has at most $m$ elements, we say that $S$ is a maximal $\epsilon$-distinguishable subset of $A$ in $X$. A maximal $\epsilon$-distinguishable subset of $A$ in $X$ is a $\epsilon$-net for $A$ in $X$.

Let $X, Y, Z, Y_{1}$ be Banach spaces and $T,S\in \mathcal{L}(X,Y)$,  $R\in \mathcal{L}(Y,Z)$. Then for all $k,l \in\N$
(see \cite{Edmunds} or \cite{Pisier})
\begin{equation}\label{ecuaentroart2}
e_{k+l-1}\left(T+S\right)\leq e_{k}\left(T\right)+e_{l}\left(S\right) 
\end{equation}
and
\begin{equation}\label{multiplicative}
e_{k+l-1}\left(R\circ S\right)\leq e_{k}\left(R\right) e_{l}\left(S\right).
\end{equation}
Suppose $Y$ is isometric to a subspace of $Y_{1}$ and denote by $i:Y\to Y_{1}$ the isometric embedding. Then (see \cite[Proposition 5.1]{Pisier})
\begin{equation}\label{ecuaentroart1}
2^{-1}e_{k}\left(T\right)\leq e_{k}\left(i\circ T\right)\leq e_{k}\left(T\right), \hspace{1ex} k\in\N.
\end{equation}
\begin{remark}
 For $d\in\mathbb{N}$ and $R>0$ let $B_R=\left\{\textup{\textbf{x}}\in \mathbb{R}^d : \abs{\textup{\textbf{x}}} \leq R\right\}$. The number of points of integer coordinates contained in the ball $B_R$ is given by 
$$
N_d(R)=\left(\dfrac{2\pi^{d/2}}{d\varGamma(d/2)} \right)R^d+ E_d(R),
$$
where $E(R)$ is an error term. We denote 
$$ 
{\theta}_d = \inf \left\{ \alpha : E_d(R)={\cal O}(R^{\alpha})\right\}.
$$
For $d=1$ we have ${\theta}_1=0$ and it is known that for $d\geq 4$, ${\theta}_d=d-2$. The value of ${\theta}_d$ remains open for
$d=2$ and $d=3$. It is known that $1/2\leq {\theta}_2 \leq 1$ and it was proved in \cite{Hu} that, for each ${\theta} > 131/208$, there is a constant $C>0$ such that $E_{2}(R)\leq CR^{\theta}$, for all $R>0$. In the case $d=3$ we have $1\leq {\theta}_3 \leq 2$ and the estimate  $E_3(R)\leq CR^{\theta}$ is true for  ${\theta}=21/16$ (see \cite{HB}).

As consequence of the estimates for the number of points with integer coordinates in $B_{R}$, there is a constant $C_d$, depending only on $d$, such that, for all $l\in \mathbb{N}$,
\begin{eqnarray}\label{cardinalidade-1}
 &&2^{-d} N_d(l) \leq \# A_l \leq 2^{-d}N_d(l) + d2^{-d} N_{d-1}(l) + C_d l^{d-2},\\
\vspace{15ex}
 &&2^{-d} \left(N_d(l) - N_d(l-1) \right) \leq  \# A_l - \# A_{l-1} \nonumber\\ 
  \label{cardinalidade-2} &&\leq 2^{-d} \left(N_d(l) - N_d(l-1) \right) + d2^{-d} \left(N_{d-1}(l) - N_{d-1}(l-1) \right) + C_dl^{d-2}.
\end{eqnarray}
\end{remark}

\begin{proposition}\label{proposicionaplicaciones1}
Given $d \geq 2$, let $\theta= 132/208$ if $d=2$,  $\theta= 21/16$ if $d=3$ and $\theta=d-2$ if $d\geq 4$.
Then there are positive constants $C^{'}$ and $C$, such that, for all $l, N\in \mathbb{N}$,
\begin{equation}\label{contagem-1}
\frac{2^{1-d} \pi^{d/2}}{\varGamma(d/2)}l^{d-1}- C^{'}l^{\theta} \leq d_l \leq
\frac{2^{1-d} \pi^{d/2}}{\varGamma(d/2)}l^{d-1}+ C^{'}l^{\theta},
\end{equation}
\begin{equation}\label{contagem-2}
\frac{2^{1-d} \pi^{d/2}}{d\varGamma(d/2)}N^{d} \leq \dim{{\cal T}_N} \leq \frac{2^{1-d} \pi^{d/2}}{d\varGamma(d/2)}N^{d}+CN^{d-1}
\end{equation}
and
\begin{equation} \label{contagem-3}
\frac{1}{\dim{{\cal T}_N}}\geq \frac{1}{FN^d}-\frac{C}{F^2N^{d+1}},
\end{equation}
where $F={2^{1-d} \pi^{d/2}}/\left({d\varGamma(d/2)}\right)$. In particular, $d_l\asymp l^{d-1}$ and $\dim{{\cal T}_N}\asymp N^d$.
\end{proposition}
\begin{proof}
We have
\begin{equation*}
N_d(l)-N_d(l-1)
= \frac{2 \pi^{d/2}}{d\varGamma(d/2)}(l^d - (l-1)^{d})+(E_d(l) - E_d(l-1))
\end{equation*}
and
$$ dl^{d-1} - \frac{d(d-1)}{2}l^{d-2}\leq l^d - (l-1)^{d} \leq dl^{d-1},$$
therefore
$$
\frac{2 \pi^{d/2}}{\varGamma(d/2)}l^{d-1} - \frac{ \pi^{d/2}(d-1)}{\varGamma(d/2)}l^{d-2}  + E_d(l) - E_d(l-1) \leq  
N_d(l) - N_d(l-1).
$$
By hypothesis $|E_d(l) - E_d(l-1)| \leq C_1l^{\theta}$ and thus
\begin{equation}\label{r2.2}
\frac{2 \pi^{d/2}}{\varGamma(d/2)}l^{d-1}-C_2l^{\theta} \leq N_d(l) - N_d(l-1) 
\leq \frac{2 \pi^{d/2}}{\varGamma(d/2)}l^{d-1} + C_1 l^{\theta},
\end{equation}
since $d-2 \leq \theta < d-1$. Then, using (\ref{r2.2}) and (\ref{cardinalidade-2}) we get
\begin{equation*}\label{r2.3}
\frac{2^{1-d} \pi^{d/2}}{\varGamma(d/2)}l^{d-1}-C_3l^{\theta} \leq \# A_l - \# A_{l-1}=d_l.
\end{equation*}
Again using (\ref{r2.2}) and (\ref{cardinalidade-2}),
\begin{eqnarray*} \label{r2.4}
d_l& \leq& 2^{-d} \left(N_d(l) - N_d(l-1) \right) + d2^{-d} \left(N_{d-1}(l) - N_{d-1}(l-1) \right) + C_dl^{d-2}\nonumber \\
&\leq& \frac{2^{1-d} \pi^{d/2}}{\varGamma(d/2)}l^{d-1} + 2^{-d}C_1 l^{\theta} +
\frac{d2^{1-d} \pi^{(d-1)/2}}{\varGamma((d-1)/2)}l^{d-2} + C_4 l^{\alpha} + C_dl^{d-2} \nonumber\\
&\leq& \frac{2^{1-d} \pi^{d/2}}{\varGamma(d/2)}l^{d-1} + C_5 l^{\theta}, 
\end{eqnarray*}
where $d-3 \leq {\theta}_{d-1}< \alpha <d-2$.
Hence the estimates in (\ref{contagem-1}) are proved.
 
On the other hand, since $E_d(N) \leq C_6 N^{\theta} \leq C_6 N^{d-1}$,
$N_d(N)=\left({2 \pi^{d/2}}/({d\varGamma(d/2)})\right) N^d+ E_d(N)$ and  $\dim{{\cal T}_N}= \# A_N$, it 
follows by (\ref{cardinalidade-1}) that $
\left(2^{1-d} \pi^{d/2}/d\varGamma(d/2) \right)N^d \leq \dim{{\cal T}_N}$ and
\begin{eqnarray*}
\dim{{\cal T}_N} &\leq&  \left(\dfrac{2^{1-d} \pi^{d/2}}{d\varGamma(d/2)} \right)N^d + 2^{-d}C_7N^{\theta} + 
\left(\dfrac{d2^{1-d} \pi^{(d-1)/2}}{(d-1)\varGamma((d-1)/2)} \right)N^{d-1} + C_{8} N^{d-2}\\
&\leq& \left(\dfrac{2^{1-d} \pi^{d/2}}{d\varGamma(d/2)} \right)N^d + C_{9}N^{d-1},
\end{eqnarray*}
and we get (\ref{contagem-2}).

The estimate (\ref{contagem-3}) follows immediately from (\ref{contagem-2}). As consequence of (\ref{contagem-1}) and (\ref{contagem-2}) we have that  $d_l\asymp l^d$ and $\dim{{\cal T}_N}\asymp N^d$.
\end{proof}

The Proposition \ref{proposicionaplicaciones1} is applied in the proof of Theorems \ref{aplicacionesentropiafinitas1} and \ref{teoremainfinitasentropia}. To prove these theorems we use all the estimates (\ref{contagem-1}), (\ref{contagem-2}) and  (\ref{contagem-3}).
The constants in these estimates are not important in the proof of Theorem \ref{aplicacionesentropiafinitas1}, but they are fundamental in the proof of Theorem \ref{teoremainfinitasentropia}. The constant $\mathcal{C}$ in Theorem \ref{teoremainfinitasentropia}, with which we obtain sharp estimates for entropy numbers, is constructed using these constants. Also the 
values of $\theta$ determined in the proposition play an important role in Theorem \ref{teoremainfinitasentropia}.

Let $\Lambda=\lbrace \lambda_{\textbf{k}}\rbrace_{\textbf{k}\in \N^{d}_{0}}$, $\lambda_{\textbf{k}}\in \R$, and $1\leq p,q \leq \infty$. If for all $\varphi \in L^{p}(\Omega)$ there is a function $f=\Lambda \varphi \in L^{q}(\Omega)$ with formal Fourier expansion given by
\[
f\sim \sum_{\textbf{k}\in \N^{d}_{0}}\lambda_{\textbf{k}}\widehat{\varphi}(\textbf{k})\phi_{\textbf{k}},
\]
such that $\norm{\Lambda}_{p,q}:=\sup\lbrace \norm{\Lambda\varphi}_{q}: \varphi\in U_{p}\rbrace<\infty$, we say that the multiplier operator $\Lambda$ is bounded from $L^{p}(\Omega)$ into $L^{q}(\Omega)$, with norm $\norm{\Lambda}_{p,q}$. Let $\lambda:[0,\infty)\to \R$ be a function and for every $\textbf{k}\in \N^{d}_{0}$ let $
\lambda_{\textbf{k}}=\lambda\left(\abs{\textbf{k}}\right), $ $\lambda_{\textbf{k}}^{\ast}=\lambda\left(\abs{\textbf{k}}_{\ast}\right)$. In this paper, we will consider only multiplier operators associated with sequences of the type $\Lambda=\left\lbrace \lambda_{\textbf{k}}\right\rbrace_{\textbf{k}\in\N^{d}_{0}}$ and $\Lambda_{\ast}=\left\lbrace \lambda_{\textbf{k}}^{\ast}\right\rbrace_{\textbf{k}\in\N^{d}_{0}}$.

Let us write $\norm{x}_{2}=\left(\sum_{i=1}^{n}\abs{x_{i}}^{2}\right)^{1/2}$, for the euclidean norm of the element $x=(x_{1},\ldots,x_{n})\in \R^{n}$ and by $\mathbb{S}^{n-1}$ the unit euclidean sphere $\left\{x\in \R^{n}: \norm{x}_{(2)}=1\right\}$ in $\R^{n}$.
The Levy mean of a norm $\norm{\cdot}$ on $\R^{n}$ is defined by 
\[
M(\norm{\cdot}):=\Biggl(\int_{\mathbb{S}^{n-1}}\norm{x}^{2} d\sigma(x) \Biggl)^{1/2},
\]
where $d\sigma(x)$ denotes the normalized Lebesgue measure on $\mathbb{S}^{n-1}$.

Given $M_{1},M_{2}\in \N$, with $M_{1}<M_{2}$, we will use the following notations:
\[
\mathcal{T}_{M_{1},M_{2}}=\bigoplus_{l=M_{1}+1}^{M_{2}}\mathcal{H}_{l} \hspace{2ex} \textup{ and } \hspace{2ex} n=\dim \mathcal{T}_{M_{1},M_{2}}.
\]
\begin{remark}Let $A_{l}\setminus A_{l-1}=\left\{ \textup{\textbf{m}}^{l}_{j}:1\leq j\leq d_{l}\right\}$ where the elements $\textup{\textbf{m}}^{l}_{j}$ are chosen satisfying $\left\vert\textup{\textbf{m}}^{l}_{j}\right\vert\leq \left\vert\textup{\textbf{m}}^{l}_{j+1}\right\vert$ for $1\leq j\leq d_{l}-1$. Then $\left\lbrace \phi_{\textup{\textbf{m}}^{l}_{j}}:1\leq j\leq d_{l}\right\rbrace$ is an orthonormal basis of $\mathcal{H}_{l}$. We consider the orthonormal basis 
\[
\Upsilon=\Upsilon_{M_{1},M_{2}}=\left\lbrace\phi_{j}^{l}=\phi_{\textup{\textbf{m}}^{l}_{j}}:M_{1}+1\leq l\leq M_{2},1\leq j\leq d_{l}\right\rbrace
\]
of $\mathcal{T}_{M_{1},M_{2}}$ endowed with the order $\phi_{1}^{M_{1}+1},\ldots,\phi_{d_{M_{1}+1}}^{M_{1}+1},\phi_{1}^{M_{1}+2},\ldots,\phi_{d_{M_{1}+2}}^{M_{1}+2}, \ldots,\phi_{1}^{M_{2}},\ldots,\phi_{d_{M_{2}}}^{M_{2}}$. We denote 
\[
\xi_{k}=\phi_{t}^{l+1},\hspace{2ex} k=t+\sum_{j=M_{1}+1}^{l}d_{j},\hspace{2ex} 1\leq t\leq d_{l+1},\hspace{2ex} M_{1}\leq l< M_{2},
\]
and hence $\Upsilon=\lbrace \xi_{k}\rbrace_{k=1}^{n}$. Let $J:\R^{n}\to \mathcal{T}_{M_{1},M_{2}}$ be the coordinate isomorphism that assigns to \linebreak $\alpha=(\alpha_{1},\alpha_{2},\ldots,\alpha_{n})\in \R^{n}$ the function
\[
J(\alpha)=J(\alpha_{1},\ldots,\alpha_{n})=\sum_{k=1}^{n}\alpha_{k}\xi_{k}\in \mathcal{T}_{M_{1},M_{2}}.
\]
Consider a function $\lambda: [0,\infty)\to \R$ such that, $\lambda(t)\neq 0$, for $t\geq 0$, and let $\Lambda=\lbrace \lambda_{\textup{\textbf{k}}}\rbrace_{\textup{\textbf{k}}\in \N_{0}^{d}}$ be the sequence of multipliers defined by $\lambda_{\textup{\textbf{k}}}=\lambda(\abs{\textup{\textbf{k}}})$. Consider $\left\lbrace\lambda^{l}_{j}=\lambda_{\textup{\textbf{m}}^{l}_{j}}: M_{1}+1\leq l\leq M_{2},1\leq j\leq d_{l}\right\rbrace$ endowed with the order $\lambda_{1}^{M_{1}+1},\ldots,\lambda_{d_{M_{1}+1}}^{M_{1}+1}$, $\lambda_{1}^{M_{1}+2},\ldots,\lambda_{d_{M_{1}+2}}^{M_{1}+2},\ldots,\lambda_{1}^{M_{2}},\ldots,\lambda_{d_{M_{2}}}^{M_{2}}$. We denote
\[
\lambda_{k}=\lambda^{l+1}_{t}=\lambda\left(\left\vert\textup{\textbf{m}}^{l+1}_{t}\right\vert\right),\hspace{1ex} k=t+\sum_{j=M_{1}+1}^{l}d_{j},\hspace{1ex}1\leq t\leq d_{l+1},\hspace{1ex}M_{1}\leq l < M_{2}
\]
and $\Lambda_{n}=\lbrace \lambda_{k}\rbrace_{k=1}^{n}$. Now consider the multiplier operator $\Lambda_{n}$ defined on $\mathcal{T}_{M_{1},M_{2}}$ by
\begin{equation}\label{operador1}
\Lambda_{n}\left (\sum_{j=1}^{n} \alpha_{j}\xi_{j} \right)=\sum_{j=1}^{n}\lambda_{j}\alpha_{j}\xi_{j}.
\end{equation}
Also, we define the multiplier operator $\widetilde{\Lambda}_{n}$ on $\R^{n}$ by
\begin{equation}\label{operador2}
\widetilde{\Lambda}_{n}(\alpha)=\widetilde{\Lambda}_{n}(\alpha_{1},\ldots,\alpha_{n})=(\lambda_{1}\alpha_{1},\ldots,\lambda_{n}\alpha_{n}).
\end{equation}
Given $\alpha\in\R^{n}$ and $1\leq p\leq \infty$, we define
\[
\norm{\alpha}_{(p)}:=\norm{J(\alpha)}_{p}
\]
and we have that the map $\R^{n}\ni \alpha\to \norm{\alpha}_{(p)}$ is a norm on $\R^{n}$. We will denote
\[
B^{n}_{p}=\lbrace \varphi\in\mathcal{T}_{M_{1},M_{2}}: \norm{\varphi}_{p}\leq 1\rbrace,
\hspace{2ex}
B^{n}_{(p)}=\lbrace \alpha\in\R^{n}: \norm{\alpha}_{(p)}\leq 1\rbrace.
\]

\end{remark}

The following result is a simple generalization of a particular case of the Theorem 5, p. 1301 in \cite{Sergio}.

\begin{theorem}\label{teorema-media-1}
Let $n=\dim \mathcal{T}_{M_{1},M_{2}}$, $\Upsilon_{M_{1},M_{2}}=\lbrace\xi_{k}\rbrace_{k=1}^{n}$ be the orthonormal system  of $\mathcal{T}_{M_{1},M_{2}}$. Then there is an absolute constant $C>0$ such that, for $2\leq p\leq \infty$,
\begin{equation*}
    1\leq M\left(\norm{\cdot}_{(p)}\right)\leq C  \left\{  \begin{array}{ll}
p^{1/2},&\textbf{$2\leq p <\infty $},\\
(\log_{2} n)^{1/2},& \textbf{$ p=\infty$}.
\end{array}
\right.
\end{equation*}
\end{theorem}

\section{Estimates for entropy numbers of general multiplier operators}

Let $K$ be a compact set in $\R^{n}$ and let us denote by $\langle x,y \rangle$ the usual inner product of $x,y \in \R^{n}$. The polar set of $K$ is the set 
\[
K^{\circ}=\left\lbrace x\in\R^{n}: \sup_{y\in K}\abs{\langle x,y \rangle}\leq 1\right \rbrace,
\]
and we define 
\[
\norm{x}_{K^{\circ}}=\sup\left\lbrace\abs{\langle x,y\rangle}: y\in K\right\rbrace, \hspace{2ex} x\in\R^{n}.
\]

\begin{theorem}(Urysohn's inequality, \cite{Pisier}, p. 6)\label{desigualdadeurysohn}
Let $K$ be a compact set in $\R^{n}$. Then
\[
\left( \frac{\textup{Vol}_{n}(K)}{\textup{Vol}_{n}\left(B^{n}_{(2)}\right)}\right)^{1/n}\leq \int_{\mathbb{S}^{n-1}}\norm{x}_{K^{\circ}}d\sigma (x),
\]
where $\textup{Vol}_{n}(A)$ denote the volume of a measurable subset $A$ in $\R^{n}$.
\end{theorem}
\begin{proposition}\label{proposicionentropias1} (\cite{Pisier}, p. \textup{xi})
Let $V$ be a convex, centrally symmetric, bounded and absorbent subset of $\R^{n}$. Then there is a constant $C>0$, not dependent on $n$, such that for all $n\in \N$,
\[
\left( \frac{\textup{Vol}_{n}(V)\textup{Vol}_{n}\bigl(V^{\circ}\bigl)}{\left(\textup{Vol}_{n}\left(B^{n}_{(2)}\right)\right)^{2}}\right)^{1/n}\geq C.
\]
\end{proposition} 
\begin{theorem}\label{teoremainferioressentropia}
Let $\Lambda=\lbrace\lambda_{\textup{\textbf{k}}}\rbrace_{\textup{\textbf{k}}\in \N^{d}_{0}}$ be an arbitrary multiplier operator, where $\lambda_{\textup{\textbf{k}}}=\lambda(\abs{\textup{\textbf{k}}})$, for some function $\lambda:[0,\infty)\to \R$, such that $t\mapsto\abs{\lambda(t)}$  is  non-increasing. Then there is a constant $C>0$, depending only on $p$ and $q$, such that, for all $N,k\in \N$ and $n=\dim\mathcal{T}_{N} =\sum_{l=1}^{N}d_{l}$, $d_{l}=\dim \mathcal{H}_{l}$, we have
\[
e_{k}\left(\Lambda U_{p},L^{q}\right)\geq C2^{-k/n}\left (\prod^{N}_{l=1}\abs{\lambda(l)}^{d_{l}} \right)^{1/n} \mathcal{V}_{n},
\] 
where 
  \begin{equation*}
   \mathcal{V}_{n}= \left\{
\begin{array}{ll}
1,&\textbf{$ p <\infty$, $ q>1$},\\
(\log_{2}n)^{-1/2},&\textbf{$ p< \infty$, $q=1$},\\
(\log_{2}n)^{-1/2},&  p=\infty,q>1,\\
(\log_{2} n)^{-1},& \textbf{$ p=\infty$, $q=1$}.
\end{array}
\right.
   \end{equation*}
   In particular, if $k=n$, then
   \begin{equation}\label{ecuacionentropianueva1}
  e_{n}\left(\Lambda U_{p},L^{q}\right)\geq C \abs{\lambda(N)}\mathcal{V}_{n}.   
   \end{equation}
\end{theorem}
\begin{proof}
If for some $1\leq l \leq N$, $\lambda(l)=0$, then the statement is trivial. Let us suppose $\lambda(l)\neq 0$ for all $1\leq l\leq N$. 
Consider a norm $\norm{\cdot}$ on $\R^{n}$ and denote by $E$ the Banach space $\left(\R^{n},\norm{\cdot}\right)$. The dual norm of $\norm{\cdot}$ is given by $\norm{x}^{\circ}=\sup\left \lbrace \abs{\langle x,y \rangle}: y\in B_{E}\right\rbrace$, $x\in \R^{n}$. The dual space $\left(\R^{n}, \norm{\cdot}^{\circ}\right)$ of $E$ will be denoted by $E^{\circ}$. Using the fact that $J$ is an isomorphism and the Hölder's inequality, for all $x\in\R^{n}$ we get
\begin{align*}
    \norm{x}_{(p)}^{\circ} & =\sup \left\{\abs{\dual{x}{y}}: y \in B_{(p)}^{n}\right\}=\sup \left \{\abs{\int_{\Omega}J(x)J(\bar{y})d\mu }: J(\bar{y}) \in B_{p}^{n}\right \}\\
    &\leq\sup \left\lbrace\norm{J(x)}_{p'}\norm{J(\bar{y})}_{p}: J(\bar{y})\in B_{p}^{n}\right\rbrace \leq\norm{J(x)}_{p'}=\norm{x}_{(p')},
  \end{align*}
for any $1\leq p,p'\leq\infty$, such that $1/p+1/p'=1$. Consider $1\leq q\leq 2$ and $\norm{\cdot}=\norm{\cdot}_{(q)}$, then from Theorems \ref{desigualdadeurysohn}, \ref{teorema-media-1} and the last estimate, it follows that
\begin{align}\label{ecuacionentropia1}
\left(\frac{\textup{Vol}_{n}\left(B^{n}_{(q)}\right)}{\textup{Vol}_{n}\left(B^{n}_{(2)}\right)} \right)^{1/n}&\leq \int_{\mathbb{S}^{n-1}}\norm{x}^{\circ}_{(q)}d\sigma(x)\leq \int_{\mathbb{S}^{n-1}}\norm{x}_{(q')}d\sigma(x)\leq \left(\int_{\mathbb{S}^{n-1}}\norm{x}^{2}_{(q')}d\sigma(x) \right)^{1/2}\nonumber\\
&=M\left(\norm{\cdot}_{(q')}\right)\leq  C_{1}\left\{
\begin{array}{ll}
(q')^{1/2},&\textbf{$ 2\leq q'< \infty $},\\
(\log_{2}n)^{1/2},&\textbf{$q'=\infty$},
\end{array}
\right.
\end{align}
where $1/q+1/q'=1$. Analogously, if we take $2\leq p\leq \infty$ and $\norm{\cdot}=\norm{\cdot}_{(p')}$, we get
\begin{equation}\label{ecuacionentropia2}
\left(\frac{\textup{Vol}_{n}\left(\left(B^{n}_{(p)}\right)^{\circ}\right)}{\textup{Vol}_{n}\left(B^{n}_{(2)}\right)} \right)^{1/n}\leq M\left(\norm{\cdot}_{(p)}\right)\leq  C_{2}\left\{
\begin{array}{ll}
p^{1/2},&\textbf{$ 2\leq p< \infty$},\\
(\log_{2}n)^{1/2},&\textbf{$p=\infty$}.
\end{array}
\right.
\end{equation}
From Proposition \ref{proposicionentropias1}, there is an absolute constant $C_{3}>0$ such that
\[
\left(\frac{\textup{Vol}_{n}\left(B^{n}_{(p)}\right)\textup{Vol}_{n}\left(\left(B^{n}_{(p)}\right)^{\circ}\right)}{\textup{Vol}_{n}\left(B^{n}_{(2)}\right)^{2}} \right)^{1/n}\geq C_{3},
\]
and by (\ref{ecuacionentropia2}) we have that
\begin{equation}\label{ecuacionentropia3}
\left(\frac{\textup{Vol}_{n}\left(B^{n}_{(p)}\right)}{\textup{Vol}_{n}\left(B^{n}_{(2)}\right)} \right)^{1/n}\geq C_{3} \left(\frac{\textup{Vol}_{n}\left(B^{n}_{(p)}\right)^{\circ}}{\textup{Vol}_{n}\left(B^{n}_{(2)}\right)} \right)^{-1/n}\geq C_{4}\left\{
\begin{array}{ll}
p^{-1/2},&\textbf{$ 2\leq p<\infty$},\\
(\log_{2}n)^{-1/2},&\textbf{$p=\infty$}.
\end{array}
\right.
\end{equation}
On the other hand, let $\left\{x_{1},x_{2},\ldots,x_{N(\epsilon)}\right\}$ be a minimal $\epsilon$-net for the set $\widetilde{\Lambda}_{n}B_{(p)}^{n}$ in the Banach space $\left(\R^{n},\norm{\cdot}_{(q)}\right)$. Then
\[
\widetilde{\Lambda}_{n}B_{(p)}^{n} \subset \bigcup_{k=1}^{N(\epsilon)}\left(x_{k}+ \epsilon B^{n}_{(q)}\right).
\]
Comparing volumes, we get
\[
\textup{Vol}_{n}\left(\widetilde{\Lambda}_{n}B_{(p)}^{n} \right)\leq \epsilon^{n} N(\epsilon)\textup{Vol}_{n}\left(B^{n}_{(q)}\right)
\]
and hence
\begin{equation}\label{ecuacionentropia4}
\left(\frac{\textup{Vol}_{n}\left(\widetilde{\Lambda}_{n}B^{n}_{(p)}\right)}{\textup{Vol}_{n}\left( B^{n}_{(2)}\right)} \right)^{1/n}\leq \epsilon (N(\epsilon))^{1/n} \left(\frac{\textup{Vol}_{n}\left(B^{n}_{(q)}\right)}{\textup{Vol}_{n}\left(B^{n}_{(2)}\right)} \right)^{1/n}.
\end{equation}
Since $\textup{Vol}_{n}\left(\widetilde{\Lambda}_{n} B^{n}_{(p)}\right)=\left(\det \widetilde{\Lambda}_{n}\right)\textup{Vol}_{n}\left(B^{n}_{(p)}\right)$, from (\ref{ecuacionentropia3}), (\ref{ecuacionentropia4}), and (\ref{ecuacionentropia1}), it follows that
\begin{align}\label{ecuacionentropia6}
C_{4}\left(\det \widetilde{\Lambda}_{n}\right)^{1/n}&\left\{
\begin{array}{ll}
p^{-1/2},&\textbf{$ 2\leq p<\infty$},\\
(\log_{2}n)^{-1/2},&\textbf{$p=\infty$}.
\end{array}
\right.\leq \left(\det \widetilde{\Lambda}_{n}\right)^{1/n}\left(\frac{\textup{Vol}_{n}\left(B^{n}_{(p)}\right)}{\textup{Vol}_{n}\left( B^{n}_{(2)}\right)} \right)^{1/n}\nonumber\\
&= \left(\frac{\textup{Vol}_{n}\left(\widetilde{\Lambda}_{n}B^{n}_{(p)}\right)}{\textup{Vol}_{n}\left( B^{n}_{(2)}\right)} \right)^{1/n}\leq \epsilon \left(N(\epsilon)\right)^{1/n} \left(\frac{\textup{Vol}_{n}\left(B^{n}_{(q)}\right)}{\textup{Vol}_{n}\left( B^{n}_{(2)}\right)} \right)^{1/n}\nonumber\\
&\leq C_{1}\epsilon \left(N(\epsilon)\right)^{1/n} \left\{
\begin{array}{ll}
(q')^{1/2},&\textbf{$ 2\leq q'< \infty$},\\
(\log_{2}n)^{1/2},&\textbf{$q'=\infty$}.
\end{array}
\right.
\end{align}
We put $N(\epsilon)=2^{k-1}$ and by the definition of entropy number we get 
\[
\epsilon=e_{k}\left(\Lambda_{n}B^{n}_{(p)},\left(\R^{n},\norm{\cdot}_{(q)}\right)\right)=e_{k}\bigl (\Lambda U_{p}\cap \mathcal{T}_{N}, L^{q}\cap \mathcal{T}_{N}\bigl ),
\]
 and by (\ref{ecuaentroart1}),
\begin{equation*}
e_{k}\left(\Lambda U_{p},L^{q}\right)\geq e_{k}\left(\Lambda U_{p}\cap \mathcal{T}_{N},L^{q}\right)\geq 2^{-1}e_{k}(\Lambda U_{p}\cap \mathcal{T}_{N},L^{q}\cap \mathcal{T}_{N}).
\end{equation*}
Hence for $2\leq p <\infty$ and $1<q\leq2$, by (\ref{ecuacionentropia6}) we get
\begin{equation*}
e_{k}\left(\Lambda U_{p},L^{q}\right)\geq C_{5} 2^{-k/n}\left(\det \widetilde{\Lambda}_{n}\right)^{1/n}\geq C_{5} 2^{-k/n}\left(\prod_{l=1}^{N}\lambda(l)^{d_{l}}\right)^{1/n},
\end{equation*}
where the last inequality is true since $\lambda(l)\leq \lambda (\abs{\textbf{k}})\leq \lambda(l-1)$ for $\textbf{k}\in A_{l}\setminus A_{l-1}$. Now consider $1\leq p<2$ and $1<q\leq 2$. From the previous case
\begin{equation*}
e_{k}\left(\Lambda U_{p},L^{q}\right)\geq e_{k}(\Lambda U_{2},L^{q})\geq C 2^{-k/n}\left(\prod_{l=1}^{N}\lambda(l)^{d_{l}}\right)^{1/n}.
\end{equation*}
If $1\leq p<\infty$ and $2<q\leq \infty$, since $U_{q}\subset U_{2}$, from the previous cases we get
\begin{equation*}
e_{k}\left(\Lambda U_{p},L^{q}\right)\geq e_{k}\left(\Lambda U_{p},L^{2}\right)\geq C 2^{-k/n}\left(\prod_{l=1}^{N}\lambda(l)^{d_{l}}\right)^{1/n}.
\end{equation*}
Thus we prove the general result of the theorem for $p<\infty$ and $q>1$.
In the other cases, by similar analysis, we obtain the remaining estimates.

Since the function $t\mapsto\abs{\lambda(t)}$ is non-increasing, then \begin{equation*}
\prod_{l=1}^{N}\abs{\lambda(l)}^{d_{l}}\geq \prod_{l=1}^{N}\abs{\lambda(N)}^{d_{l}}=\abs{\lambda(N)}^{\sum_{l=1}^{N}d_{l}}=\abs{\lambda(N)}^{n}, 
\end{equation*}
and therefore putting $k=n$ we get $e_{n}\left(\Lambda U_{p},L^{q}\right)\geq C\abs{\lambda(N)}\mathcal{V}_{n}$, 
concluding the proof.
\end{proof}
	
\begin{remark}\label{observacionentropia1}
Given an arbitrary multiplier sequence $\Lambda=\left\lbrace\lambda(\abs{\textup{\textbf{k}}})\right\rbrace_{\textup{\textbf{k}} \in\N^{d}_{0}}$, for $1\leq q\leq \infty$ and $k\in\N$, we define
\[
\chi_{k}^{(q)}=\chi_{k}=3\sup_{N\geq 1}\left(\frac{2^{-k+1}\textup{Vol}_{n}\left(B^{n}_{(2)}\right)}{\textup{Vol}_{n}\left(B^{n}_{(q)}\right)}\prod^{N}_{j=1}\abs{\lambda(j)}^{d_{j}} \right)^{1/n},
\]
where $n=\dim\mathcal{T}_{N}$. We observe that $\chi_{k}$ depends on $k,q$ and $\lambda$. If $\lambda$ is a bounded function then
\[
\chi_{k}^{(q)}\leq 3\sup_{N\geq 1}\left(2^{-k+1}\right)^{1/n}\left(\frac{\textup{Vol}_{n}\left(B^{n}_{(2)}\right)}{\textup{Vol}_{n}\left(B^{n}_{(q)}\right)} \right)^{1/n}\sup_{1\leq j\leq N}\abs{\lambda(j)},
\]
for all $k\in\N$, and if $2\leq q \leq \infty$, since $B^{n}_{(q)}\subset B^{n}_{(2)}$,  it follows by (\ref{ecuacionentropia3}) that
\begin{equation}\label{ecuacionentropiasup2}
1\leq \left(\frac{Vol_{n}\left(B^{n}_{(2)}\right)}{Vol_{n}\left
(B^{n}_{(q)}\right)} \right)^{1/n}\leq C \left\{
\begin{array}{ll}
q^{1/2},&\textbf{$ 2\leq q<\infty$},\\
(\log_{2}n)^{1/2},&\textbf{$q=\infty$}.
\end{array}
\right.
\end{equation}
\end{remark}
\begin{remark}\label{observacion-superiores-1}
Consider a function  $\lambda$ such that $t\mapsto \abs{\lambda(t)}$ is a non-increasing function and $\lim_{t\to \infty}\abs{\lambda(t)}=0$. For $N\in\N$, we define $N_{1}=N$ and
\[
	N_{k+1}= \min \left\{l\in\N: 2\abs{\lambda(l)}\leq \abs{\lambda(N_{k})}\right\}, \hspace{2ex} k\in\N,
	\]
	and we have $
	\abs{\lambda(N_{k+1})}\leq 2^{-k}\abs{\lambda(N)}$. Let
	\[
	 \theta_{N_{k},N_{k+1}}= \sum_{s=N_{k}+1}^{N_{k+1}}\dim \mathcal{H}_{s}.
	\]
	Fixed a positive $\epsilon$ we write
	\[
	M= \Biggl[\frac{\log_{2}(\theta_{N_{1},N_{2}})}{\epsilon} \Biggl ], \hspace{3ex} m_{k}=[2^{-\epsilon k}\theta_{N_{1},N_{2}}]+1, \hspace{2ex} k=1,2,\ldots, M,
	\]
	and $m_{0}=\theta_{N_{0},N_{1}}=\theta_{0,N}$. We have that $\sum_{k=1}^{M}m_{k}\leq  C_{\epsilon}\theta_{N_{1},N_{2}}$, 
	where $C_{\epsilon}$ is a positive constant depending only on $\epsilon$. 	
	
	We say that  $\Lambda=\{\lambda_{\textup{\textbf{k}}}\}_{\textup{\textbf{k}}\in\N^{d}_{0}}\in K_{\epsilon,p}$, for $\epsilon>0$ and $1\leq p\leq 2$, if  for all $N\in \N$,
	\[
	\sum_{k=1}^{M}2^{-k(1-\epsilon/2)}\frac{\theta_{N_{k},N_{k+1}}^{1/p}}{\theta_{N_{1},N_{2}}^{1/2}}\leq C_{\epsilon,p}\theta_{N_{1},N_{2}}^{1/p-1/2}.
	\]
	\end{remark}
	\begin{remark}\label{Remark36}
	 Estimates of $n$-widths for the multiplier operators $\Lambda^{(1)}$ and $\Lambda^{(2)}$ were studied in \cite{Sergio}. On page 1313 it was shown that $\Lambda^{(1)}\in K_{\epsilon,2}$ if $0<\epsilon<2-d/\gamma$, $\gamma>d/2$, and on page 1318 that $\Lambda^{(2)}\in K_{\epsilon,2}$ if $0<\epsilon<\left(2r-\delta(d-r)\right)/r$, $0<\delta<2r/(d-r)$. 
	 
	 For both operators, it was proved on pages 1313 and 1319, that, for $2\leq p\leq \infty$ and $2\leq q<\infty$,
	 \[
	 \sum_{k=M+1}^{\infty}2^{-k}\theta_{N_{k},N_{k+1}}^{1/p-1/q}\leq C,
	 \]
	 where the constant $C$ does not depend on $N$. Also, for the operator $\Lambda^{(1)}$, on page 1312, the following estimates were proved,
	 \[
	 \theta_{N_{k},N_{k+1}}\asymp N^{d}_{k+1}, \hspace{2ex} 2^{k/(\gamma+1)}N\leq N_{k+1}\leq C2^{k/\gamma}N, \hspace{3ex}k\geq 1,
	 \]
	 and for the operator $\Lambda^{(2)}$, on page 1318, the following estimate can be obtained,
	 \[
	 \theta_{N_{k},N_{k+1}}\leq C 2^{(\delta(d-r)/r)k}N^{d-r},\hspace{3ex}k\geq 1.
	 \]
	 The constant $C$ does not depend on $N$.
	\end{remark}
\begin{lemma}(\cite{Pajor})\label{lemaentropia}
Let $E=(\R^{n},\norm{\cdot})$ be a $n$-dimensional Banach space  and
\[
M^{*}=M^{*}(E)=\int_{\mathbb{S}^{n-1}}\norm{x}d\sigma(x).
\]
Then, for all $m\in\N$ 
\[
e_{m}\left(B^{n}_{(2)},E\right)\ll M^{*}(E) \left\{
\begin{array}{ll}
(n/m)^{1/2},&\textbf{$ m\leq n$},\\
e^{-m/n},&\textbf{$m>n$}.
\end{array}
\right.
\]
\end{lemma}
\begin{theorem}\label{teoremasuperioresentropia}
Let $\lambda:[0,\infty)\to\R$ such that $t\mapsto \abs{\lambda(t)}$ is a non-increasing function, $\lim_{t\to \infty}\abs{\lambda(t)}=0$ and consider $\Lambda=\lbrace\lambda_{\textup{\textbf{k}}}\rbrace_{\textup{\textbf{k}}\in\N^{d}_{0}}$, $\lambda_{\textup{\textbf{k}}}=\lambda(\abs{\textup{\textbf{k}}})$. Suppose $\Lambda\in K_{\epsilon,2}$, for some $\epsilon>0$. Let $\chi_{k}$ be as in Remark \ref{observacionentropia1} and $M, N_{l},  \theta_{N_{l},N_{l+1}}$ and  $m_{l}$ as in Remark \ref{observacion-superiores-1} and $\eta=k+\sum^{M}_{l=1}m_{l}$, where $k\in\N$. Consider fixed $b\in\N_{0}$ and
let $k,N \in \N$ satisfying $\abs{\lambda(N)}\geq 2^{-b} \chi_{k}$. Then there exists an absolute constant $C>0$ such that, for $2\leq p\leq \infty$, $1\leq q\leq 2$, we have that
\[
e_{\eta+bn}\left(\Lambda U_{p}, L^{q}\right)\leq C \abs{\lambda(N)},
\]
and for $2\leq p,q\leq \infty$ 
		\begin{align}\label{ecuacionentropiasup3}
		e_{\eta+bn}(\Lambda U_{p},L^{q})&\leq  C\abs{\lambda(N)}\left(\left\{
		\begin{array}{ll}
			q^{1/2},&\textbf{$2\leq q<\infty$},\\
			\sup_{1\leq j\leq M}(\log_{2}\theta_{N_{j},N_{j+1}})^{1/2},&\textbf{$q=\infty$},
		\end{array}
		\right\} \right. \nonumber\\
		&+  \left.\sum_{j=M+1}^{\infty}2^{-j}\theta^{1/2-1/q}_{N_{j},N_{j+1}} \right).
	\end{align}
\end{theorem}
\begin{proof}
It is sufficient to prove the theorem for $p=2$ and $q\geq 2$. We denote $B_{q}^{N_{k},N_{k+1}}= U_{q}\cap  \mathcal{ T}_{N_{k},N_{k+1}}$. From 
\cite[(37), p. 1309]{Sergio}, for $p=2$, we obtain
\[
\Lambda U_{2}\subset \Lambda U_{2}\cap \mathcal{T}_{N}+\bigoplus_{l=1}^{M}\abs{\lambda(N_{l})}B_{2}^{N_{l},N_{l+1}}+\bigoplus_{l=M+1}^{\infty}\abs{\lambda(N_{l})}\theta^{1/2-1/q}_{N_{l},N_{l+1}}B^{N_{l},N_{l+1}}_{q}.
\]
Using properties of entropy numbers and (\ref{ecuaentroart2}) we get
\begin{align}\label{ecuacionentropiasup5}
e_{\eta+bn}\left(\Lambda U_{2},L^{q}\right)&\leq e_{k+bn}\left(\Lambda U_{2}\cap \mathcal{T}_{N},L^{q}\cap \mathcal{T}_{N}\right)+\sum_{l=1}^{M}\abs{\lambda(N_{l})}e_{m_{l}}\left(B_{2}^{N_{l},N_{l+1}},L^{q}\cap \mathcal{T}_{N_{l},N_{l+1}}\right)\nonumber\\
		& +e_{1}\left(\bigoplus_{l=M+1}^{\infty}\abs{\lambda(N_{l})}\theta^{1/2-1/q}_{N_{l},N_{l+1}}B^{N_{l},N_{l+1}}_{q},L^{q}\cap \mathcal{T}_{N_{l},N_{l+1}} \right) .
\end{align}
and
\begin{align}\label{ecuacionentropiasup6}
&e_{1+1-1}\left(\bigoplus_{l=M+1}^{\infty}\abs{\lambda(N_{l})}\theta^{1/2-1/q}_{N_{l},N_{l+1}}B^{N_{l},N_{l+1}}_{q},L^{q}\cap \mathcal{T}_{N_{l},N_{l+1}} \right)\nonumber\\
&\leq\sum_{l=M+1}^{\infty}\abs{\lambda(N_{l})}\theta^{1/2-1/q}_{N_{l},N_{l+1}}e_{1}\left( B^{N_{l},N_{l+1}}_{q},L^{q}\cap \mathcal{T}_{N_{l},N_{l+1}}\right)=\sum_{l=M+1}^{\infty}\abs{\lambda(N_{l})}\theta^{1/2-1/q}_{N_{l},N_{l+1}}.
\end{align}
Therefore, from (\ref{ecuacionentropiasup5}) and (\ref{ecuacionentropiasup6}) it follows that
\begin{align}\label{ecuacionentropiasup7}
e_{\eta+bn}\left(\Lambda U_{2},L^{q}\right)&\leq e_{k+bn}\left(\Lambda U_{2}\cap \mathcal{T}_{N},L^{q}\cap \mathcal{T}_{N}\right)+\sum_{l=1}^{M}\abs{\lambda(N_{l})}e_{m_{l}}\left(B^{N_{l},N_{l+1}}_{2},L^{q}\cap \mathcal{T}_{N_{l},N_{l+1}}\right)\nonumber\\
		&+\sum_{l=M+1}^{\infty}\abs{\lambda(N_{l})}\theta^{1/2-1/q}_{N_{l},N_{l+1}}.
\end{align} 
We show first that
\begin{equation}\label{ecuacionentropiasup8}
e_{k+bn}\left(\Lambda U_{2}\cap \mathcal{T}_{N},L^{q}\cap \mathcal{T}_{N}\right)=e_{k+bn}\left(\widetilde{\Lambda}_{n}B^{n}_{(2)},\left(\R^{n},\norm{\cdot}_{(q)}\right)\right)\leq 2^{-b}\chi_{k}=2^{-b}\chi_{k}^{(q)}.
\end{equation}
For $q\geq 2$, we have $U_{q}\subset U_{2}$ and
\begin{equation}\label{ecuacionentropiasup9}
\abs{\lambda(N)}B^{n}_{(q)} \subset \abs{\lambda(N)}B^{n}_{(2)}\subset J^{-1}\left( \Lambda_{n} B^{n}_{2}\right)=\widetilde{\Lambda}_{n}B_{(2)}^{n}. 
\end{equation}
Consider a maximal $2^{-b}\chi_{k}$-distinguishable subset $\Theta=\lbrace z_{j}\rbrace_{1\leq j\leq m}$ of $\widetilde{\Lambda}_{n}B^{n}_{(2)}=J^{-1}\left( \Lambda_{n} B^{n}_{2}\right)$ in $\left(\R^{n},\norm{\cdot}_{(q)}\right)$, that is, $\norm{z_{i}-z_{j}}_{(q)}\geq 2^{-b}\chi_{k}$ for all $i\neq j$. By maximality, $\Theta$ is a $2^{-b}\chi_{k}$-net of $\widetilde{\Lambda}_{n}B_{(2)}^{n}$ in $\left(\R^{n},\norm{\cdot}_{(q)}\right)$. Note that the balls 
\[
z_{j}+\biggl (\frac{2^{-b}\chi_{k}}{2}\biggl) B^{n}_{(q)}, \hspace{2ex}1\leq j\leq m
\]
 are disjoint. Applying (\ref{ecuacionentropiasup9}) and the inequality $\abs{\lambda(N)}\geq 2^{-b}\chi_{k}$, we get
\begin{equation*}
\widetilde{\Lambda}_{n}B_{(2)}^{n}+\left(\frac{2^{-b}\chi_{k}}{2}\right) B^{n}_{(q)}\subset \widetilde{\Lambda}_{n}B_{(2)}^{n}+\left(\frac{\abs{\lambda(N)}}{2}B^{n}_{(q)} \right)  \subset \frac{3}{2}\widetilde{\Lambda}_{n}B_{(2)}^{n}.
\end{equation*}
Then,
\[
\bigcup^{m}_{j=1}\biggl(z_{j}+\frac{2^{-b}\chi_{k}}{2}B^{n}_{(q)}\biggl)\subset\frac{3}{2}\widetilde{\Lambda}_{n}B_{(2)}^{n},
\]
where the union is of disjoint sets. Hence, by taking  volumes
\begin{equation*}
\left(\frac{2^{-b}\chi_{k}}{2}\right)^{n}m\cdot \textup{Vol}_{n}\left(B^{n}_{(q)}\right)\leq\frac{3^{n}}{2^{n}}\left(\prod^{N}_{j=1}\abs{\lambda(j)}^{d_{j}} \right)\textup{Vol}_{n}\left(B^{n}_{(2)}\right),
\end{equation*}
and therefore 
\begin{equation}\label{ecuacionentropiasup10}
m\leq \left(\frac{3}{2^{-b}\chi_{k}}\right)^{n}\frac{\textup{Vol}_{n}\left(B^{n}_{(2)}\right)}{\textup{Vol}_{n}\left(B_{(q)}^{n}\right)}\left(\prod_{j=1}^{N} \abs{\lambda(j)}^{d_{j}}\right).
\end{equation}
By definition of $\chi_{k}$, we get
\begin{equation}\label{ecuacionentropiasup11}
\left(\frac{3}{\chi_{k}}\right)^{n}\leq 2^{k-1}\left( \frac{\textup{Vol}_{n}\left(B^{n}_{(2)}\right)}{\textup{Vol}_{n}\left(B^{n}_{(q)}\right)}\prod_{j=1}^{N} \abs{\lambda(j)}^{d_{j}}\right)^{-1}.
\end{equation}
From (\ref{ecuacionentropiasup10}) and (\ref{ecuacionentropiasup11}), it follows that $m\leq 2^{k+bn-1}$ and $m$ is the cardinality of a $2^{-b}\chi_{k}$-net of $\widetilde{\Lambda}_{n}B_{(2)}^{n}$ in $\left(\R^{n},\norm{\cdot}_{(q)} \right)$. Therefore (\ref{ecuacionentropiasup8}) follows from the definition of entropy numbers.

Now let us find an upper bound for the expression
\begin{equation}\label{ecuacionentropiasup12}
\sum_{j=1}^{M}\abs{\lambda(N_{j})}e_{m_{j}}\left(B_{2}^{N_{j},N_{j+1}},L^{q}\cap \mathcal{T}_{N_{j},N_{j+1}}\right)+ \sum_{j=M+1}^{\infty}\abs{\lambda(N_{j})}\theta^{1/2-1/q}_{N_{j},N_{j+1}}.
\end{equation}
Applying Lemma \ref{lemaentropia} to the Banach space $E_{j}=\left(\R^{\theta_{N_{j},N_{j+1}}},\norm{\cdot}_{(q)}\right)$ and $m_{j}$, remembering that \linebreak $m_{j}=[2^{-\epsilon j}\theta_{N_{1},N_{2}}]+1\geq 2^{-\epsilon j}\theta_{N_{1},N_{2}}$ and that $M^{*}(E_{j})\leq M(E_{j})$ we obtain
\[
e_{m_{j}}\left(B^{N_{j},N_{j+1}}_{2},L^{q}\cap \mathcal{T}_{N_{j},N_{j+1}}\right)=e_{m_{j}}\left(B_{(2)}^{N_{j},N_{j+1}},E_{j}\right)\ll M^{*}(E_{j})\frac{\theta^{1/2}_{N_{j},N_{j+1}}}{m_{j}^{1/2}}\ll \frac{\theta^{1/2}_{N_{j},N_{j+1}}}{\theta^{1/2}_{N_{1},N_{2}}}2^{\epsilon j/ 2}M(E_{j}).
\]
Thus by Theorem \ref{teorema-media-1}
\[
e_{m_{j}}\left(B^{N_{j},N_{j+1}}_{2},L^{q}\cap \mathcal{T}_{N_{j},N_{j+1}}\right)\ll \frac{\theta^{1/2}_{N_{j},N_{j+1}}}{\theta^{1/2}_{N_{1},N_{2}}}2^{\epsilon j/ 2}\left\{
\begin{array}{ll}
q^{1/2},&\textbf{$2\leq q<\infty$},\\
(\log_{2}\theta_{N_{j},N_{j+1}})^{1/2},&\textbf{$q=\infty$},
\end{array}
\right.
\]
and since $\lambda(N_{j})\leq \lambda(N)2^{-j+1}$ and $\Lambda \in K_{\epsilon,2}$, we get
\begin{align}\label{ecuacionentropiasup13}
&\sum_{j=1}^{M}\abs{\lambda(N_{j})}e_{m_{j}}\left(B^{N_{j},N_{j+1}}_{2},L^{q}\cap \mathcal{T}_{N_{j},N_{j+1}}\right)\nonumber\\
&\ll \abs{\lambda(N)}\sum_{j=1}^{M}2^{-j(1-\epsilon/2)}\frac{\theta^{1/2}_{N_{j},N_{j+1}}}{\theta^{1/2}_{N_{1},N_{2}}}\left\{
\begin{array}{ll}
q^{1/2},&\textbf{$2\leq q<\infty$},\\
(\log_{2}\theta_{N_{j},N_{j+1}})^{1/2},&\textbf{$q=\infty$},
\end{array}
\right.\nonumber \\
&\ll \abs{\lambda(N)}\left\{
\begin{array}{ll}
q^{1/2},&\textbf{$2\leq q<\infty$},\\
\sup_{1\leq j\leq M}(\log_{2}\theta_{N_{j},N_{j+1}})^{1/2},&\textbf{$q=\infty$},
\end{array}
\right.
\end{align}
and
\begin{equation}\label{ecuacionentropiasup14}
\sum_{j=M+1}^{\infty}\abs{\lambda(N_{j})}\theta^{1/2-1/q}_{N_{j},N_{j+1}}\leq \abs{\lambda(N)}\sum_{j=M+1}^{\infty}2^{-j+1}\theta^{1/2-1/q}_{N_{j},N_{j+1}}\ll \abs{\lambda(N)} \sum_{j=M+1}^{\infty}2^{-j}\theta^{1/2-1/q}_{N_{j},N_{j+1}}. 
\end{equation}
Therefore, we find an upper bound for (\ref{ecuacionentropiasup12}). Hence, from  (\ref{ecuacionentropiasup7}), (\ref{ecuacionentropiasup8}), (\ref{ecuacionentropiasup13}) and (\ref{ecuacionentropiasup14}), we obtain (\ref{ecuacionentropiasup3}). The first estimate follows from (\ref{ecuacionentropiasup3}) for $q=2$.
\end{proof}

\section{Proofs of Theorems \ref{aplicacionesentropiafinitas1} and \ref{teoremainfinitasentropia}}

\begin{lemma}\label{lemafinitaentropia}
	For $\gamma,\xi>0$ and $k\in\N$, let $g:[2,\infty)\to \R$ be the function defined by
	\[
	g(x)=-\frac{k}{x}-\frac{\gamma}{d}\log_{2}x
	-\xi \log_{2}(\log_{2}x).
	\]
	Then there are constants $C_{1},C_{2}, \overline{C}>0$ depending only on $\gamma,\xi$ and $d$ such that the absolute maximum value of the function $g$ is assumed at a point $x_{k}$ satisfying
	\[
	C_{1}k\leq x_{k}\leq C_{2}k, \hspace{2ex} k\geq 1
	\]
	and
	\[
	-\overline{C}-\frac{\gamma}{d}\log_{2}k-\xi \log_{2}(\log_{2}k) \leq 
	g(x_{k})\leq \overline{C}-\frac{\gamma}{d}\log_{2}k-\xi\log_{2}(\log_{2}k).
	\]
\end{lemma}
\begin{proof}
	We have that $g'(x)=0$ if and only if $x=k (\ln 2)^{2}d(\log_{2}x)/((\ln 2)(\log_{2}x)\gamma+d\xi)$.
	
	The function $g'$ is zero for only one point $x_{k}$, where the function $g$ assumes its absolute maximum value. If $h(x)=\log_{2}x/\left((\ln 2)(\log_{2}x\right)\gamma+ d\xi)$, then $h(2)\leq h(x)\leq 1/\left(\gamma \ln 2\right)$ for $x\geq 2$ and therefore 
	\begin{equation*}
	\frac{d (\ln 2)^{2}}{\gamma (\ln 2)+d\xi}k\leq k(\ln 2)^{2}d \frac{\log_{2}x}{(\ln 2)(\log_{2}x)\gamma+d\xi}\leq k \frac{(\ln 2)d}{\gamma}.
	\end{equation*}
Then, there are constants  $C_{1},C_{2}$ depending only on $d,\gamma,\xi$ such that $	C_{1}k\leq x_{k}\leq C_{2}k$, $k\geq 1$. 	Let $C_{1}\leq C\leq C_{2}$ such that $x_{k}=Ck$. Thus
	\begin{equation*}
	g(x_{k})=g(Ck)=-\frac{k}{Ck}-\frac{\gamma}{d}\log_{2}(Ck)-\xi \log_{2}(\log_{2}Ck)\leq \overline{C}-\frac{\gamma}{d}\log_{2}k-\xi \log_{2}(\log_{2}k)
	\end{equation*}
	and
	\[
	g(x_{k})\geq -C'-\frac{\gamma}{d}\log_{2}k-\xi \log_{2}(\log_{2}k).
	\]
\end{proof}
\begin{proof}[\textup{\textbf{Proof of Theorem \ref{aplicacionesentropiafinitas1}}.}]
We start proving (\ref{aplicacionfinitaentropia3}). Let $n=\dim \mathcal{T}_{N}$. Since $n\asymp N^{d}$ then $\lambda(N)=N^{-\gamma}(\log_{2}N)^{-\xi}\asymp n^{-\gamma/d}(\log_{2}n)^{-\xi}$, therefore by (\ref{ecuacionentropianueva1}) we get
	\[
	e_{n}\left(\Lambda^{(1)}U_{p},L^{q}\right)\gg n^{-\gamma/d}(\log_{2}n)^{-\xi}\mathcal{V}_{n}.
	\]
	Consider $k\in\N$, such that $\dim \mathcal{T}_{N-1}\leq k\leq \dim \mathcal{T}_{N}$. Then by the above estimate we obtain  \begin{equation*}
	e_{k}\left(\Lambda^{(1)}U_{p},L^{q}\right)\geq e_{n}\left(\Lambda^{(1)}U_{p},L^{q}\right)\\
	\gg n^{-\gamma/d}(\log_{2}n)^{-\xi}\mathcal{V}_{n}.
	\end{equation*}
	But, $n\asymp N^{d}\asymp (N-1)^{d}\asymp \dim \mathcal{T}_{N-1}\leq k\leq \dim \mathcal{T}_{N}=n$, then $k\asymp n$ and therefore we obtain (\ref{aplicacionfinitaentropia3}).
	
	We will prove (\ref{aplicacionfinitaentropia2}). Note that 
	\begin{equation}\label{aplicacionfinitaentropia4}
	\sigma_{n}=\left( \prod^{N}_{j=2}\abs{\lambda(j)}^{d_{j}} \right)^{1/n}\geq  \left( \abs{\lambda(N)}^{\sum_{j=2}^{N}d_{j}}\right)^{1/n}\geq N^{-\gamma}(\log_{2}N)^{-\xi}\asymp n^{-\gamma/d}(\log_{2}n)^{-\xi}.
	\end{equation}
	 Moreover,
	\begin{equation}\label{aplicacionfinitaentropia5}
	\log_{2}\sigma_{n}=\log_{2}\left( \prod^{N}_{j=2}\abs{\lambda(j)}^{d_{j}} \right)^{1/n}\leq -\frac{\gamma}{n}\sum_{j=3}^{N}d_{j}\log_{2}j-\frac{\xi}{n}\sum_{j=3}^{N}d_{j}\log_{2}\left(\log_{2}j\right)+C_{1}.
	\end{equation}
	From Proposition \ref{proposicionaplicaciones1}, it follows that $d_{j}\geq E j^{d-1}-C_{2}j^{\theta}$, where $E=2^{1-d}\pi^{d/2}/\Gamma(d/2)$, $d-2\leq \theta< d-1$ and $C_{2}$ is a positive constant. If $f(x)=Ex^{d-1}\log_{2}x-C_{2}x^{\theta}\log_{2}x$ and $	g(x)=Ex^{d-1}\log_{2}(\log_{2}x)-C_{2}x^{\theta}\log_{2}(\log_{2}x)$, then
	\[
	\frac{1}{n}\sum_{j=3}^{N}d_{j}\log_{2}j\geq \frac{1}{n}\int_{2}^{N}f(x)dx\geq \frac{1}{n}\left( \frac{E}{d}N^{d}\log_{2}N-C_{2}\frac{N^{\theta+1}}{\theta+1}\log_{2}N+C_{3}\right)
	\]
	and
	\[
	\frac{1}{n}\sum_{j=3}^{N}d_{j}\log_{2}(\log_{2}j)\geq \frac{1}{n}\int_{2}^{N}g(x)dx\geq \frac{1}{n} \left( \frac{E}{d}N^{d}\log_{2}(\log_{2}N)-C_{2}\frac{N^{\theta+1}}{\theta+1}\log_{2}(\log_{2}N)+C_{4}\right).
	\]
 From Proposition \ref{proposicionaplicaciones1}, it follows that $1/n\geq 1/(FN^{d})-C/(F^{2}N^{d+1})$, where $F=E/d$ and therefore
	\begin{equation}\label{ecuacionfinitaentropia6}
	\frac{1}{n}\sum_{j=3}^{N}d_{j}\log_{2}j\geq \biggl (\frac{1}{FN^{d}}-\frac{C}{F^{2}N^{d+1}} \biggl)\biggl(\frac{E}{d}N^{d}\log_{2}N-C_{2}\frac{N^{\theta+1}}{\theta+1}\log_{2}N+C_{3}\biggl)\geq \log_{2}N+C_{5},
	\end{equation}
	and in the same way we get
	\begin{equation}\label{ecuacionfinitaentropia7}
	\frac{1}{n}\sum_{j=3}^{N}d_{j}\log_{2}(\log_{2}j)\geq \log_{2}(\log_{2}N)+C_{6}.
	\end{equation}
	Thus, we obtain by (\ref{aplicacionfinitaentropia5})-(\ref{ecuacionfinitaentropia7})
	\begin{equation*}
	\log_{2}\sigma_{n}\leq -\gamma \log_{2}N-\xi \log_{2}(\log_{2}N)+C_{7}
	\end{equation*}
	and since $N^{d}\asymp n$, then
	\begin{equation}\label{ecuacionfinitaentropia8}
	\sigma_{n}\ll 2^{-\gamma \log_{2}N}2^{-\xi \log_{2}(\log_{2}N)}\asymp n^{-\gamma/d}(\log_{2}n)^{-\xi}.
	\end{equation}
	By  (\ref{aplicacionfinitaentropia4}) and (\ref{ecuacionfinitaentropia8}), we get
	\begin{equation}\label{ecuacionfinitaentropia9}
	\sigma_{n}\asymp n^{-\gamma/d}(\log_{2}n)^{-\xi}.
	\end{equation}
	Let $2\leq q <\infty$. From (\ref{ecuacionentropiasup2}) and (\ref{ecuacionfinitaentropia9}), we obtain 
	\begin{equation}\label{ecuacionfinitaentropia10}
	\chi_{k}=3\sup_{\overline{N}\geq 1}\left( \frac{2^{-k+1}\textup{Vol}_{\overline{n}}\left(B^{\overline{n}}_{(2)}\right)}{\textup{Vol}_{\overline{n}}\left(B^{\overline{n}}_{(q)}\right)}\prod^{\overline{N} }_{j=1}\abs{\lambda(j)}^{d_{j}}\right)^{1/\overline{n}}
	\asymp \sup_{\overline{N}\geq 1} 2^{-k/\overline{n}}\sigma_{\overline{n}}\asymp \sup_{\overline{N}\geq 1} 2^{-k/\overline{n}}\overline{n}^{-\gamma/d}(\log_{2}\overline{n})^{-\xi},
	\end{equation}
	where $\overline{n}=\dim \mathcal{T}_{\overline{N}}$. 	Consider 
	\[
	g(x)=-\frac{k}{n}-\frac{\gamma}{d}\log_{2}x-\xi\log_{2}(\log_{2}x).
	\]
	 Thus, by Lemma \ref{lemafinitaentropia} and  (\ref{ecuacionfinitaentropia10}), for $2\leq q <\infty$, it follows that
	\begin{equation}\label{artecua22}
	\chi_{k} \asymp \sup_{\overline{N}\geq 1}2^{g(\overline{n})}=2^{\sup_{\overline{N}\geq 1}g(\overline{n})}\asymp 2^{g(x_{k})}\asymp 2^{-(\gamma/d)\log_{2}k-\xi \log_{2}(\log_{2}k)}= k^{-\gamma/d}(\log_{2}k)^{-\xi}.
	\end{equation}

	From (\ref{artecua22}) we have that $\chi_{k}\asymp \lambda (k^{1/d})\asymp \lambda ([k^{1/d}])$ and thus there is $b\in\N_{0}$ such that $\lambda(N)\geq 2^{-b}\chi_{k}$ for $N=[k^{1/d}]$, $k\in \N$, and from Remark \ref{Remark36} we have that $\Lambda^{(1)}\in K_{\epsilon, 2}$ for $0<\epsilon< 2-d/\gamma$. Hence, applying Theorem \ref{teoremasuperioresentropia}, observing that $\sum_{j=M+1}^{\infty}2^{j}\theta^{1/2-1/q}_{N_{j},N_{j+1}}\ll 1$ (see Remark \ref{Remark36}) and from (\ref{artecua22}), for $2\leq p\leq \infty,\hspace{1ex} 2\leq q<\infty$, we get	\begin{equation}\label{ecuacionfinitaentropia11}
e_{\eta+bn}\left(\Lambda^{(1)}U_{p},L^{q}\right)\ll C\lambda(N)q^{1/2}\ll k^{-\gamma/d}(\log_{2}k)^{-\xi}.
	\end{equation}
	 Using estimates from Remark \ref{Remark36} and Remark \ref{observacion-superiores-1} we get
	\[
	\eta+bn=k+bn+\sum_{j=1}^{M}m_{j}\asymp k+N^{d} +C_{\epsilon}\theta_{N_{1},N_{2}}\asymp k+ N^{d}\asymp k,
	\]
	therefore from (\ref{ecuacionfinitaentropia11}) we get
	\begin{equation*}\label{ecuacionfinitaentropia12}
		e_{k}\left(\Lambda^{(1)}U_{p},L^{q}\right)\ll k^{-\gamma/d}(\log_{2}k)^{-\xi}q^{1/2},\hspace{2ex} 2\leq p\leq \infty,\hspace{1ex} 2\leq q<\infty.
	\end{equation*}
The estimate for $1\leq q\leq 2$ follows from the above estimate for $q=2$. For $2\leq p\leq \infty$ and $q=\infty$, using (\ref{ecuacionentropiasup2}) and proceeding in a similar way to the previous case we obtain
	\begin{equation*}\label{ecuacionfinitaentropia13}
	e_{k}\left(\Lambda^{(1)}U_{p},L^{q}\right)\ll k^{-\gamma/d}(\log_{2}k)^{-\xi}(\log_{2}k)^{1/2},
	\end{equation*}
	 concluding the proof of (\ref{aplicacionfinitaentropia2}).

Now, we will prove (\ref{applicationfinite4}). Let $X$ be a Banach space and $A \subset X$. The Kolmogorov $n$-width of $A$ in $X$ is defined by 
$$
d_n(A,X)= \inf_{X_n} \sup_{x\in A} \inf_{y\in X_n} \|x-y\|_{X},
$$
where $X_n$ runs over all subspaces of $X$ of dimension $n$, and the Gel'fand	$n$-width of $A$ in $X$ is defined by
$$
d^n(A,X)= \inf_{L^n} \sup_{x\in A\cap L^n}  \|x\|_{X},
$$
where $L^n$ runs over all subspaces of $X$ of codimension $n$. It was proved in \cite[Theorem 1]{Sergio} that for 
$\gamma >d/2$, $\xi \geq 0$ and all $n\in \mathbb{N}$,
\begin{equation}\label{finita1}
d_{n}\left(\Lambda^{(1)}U_{2},L^{q}\right)\ll n^{-\gamma/d}(\log_{2}n)^{-\xi}\left\{
\begin{array}{ll}
q^{1/2},&{2\leq q< \infty},\\
(\log_{2}n)^{1/2},& {q=\infty}.
\end{array}
\right.
\end{equation}	
From the duality of Kolmogorov and Gel'fand $n$-widths (see \cite[p. 34]{Pinkus}) it follows that
\begin{equation}\label{finita2}
d^{n}\left(\Lambda^{(1)}U_{q'},L^{2}\right)=
d_{n}\left(\Lambda^{(1)}U_{2},L^{q}\right)\ll n^{-\gamma/d}(\log_{2}n)^{-\xi}\left\{
\begin{array}{ll}
q^{1/2},&{2\leq q< \infty},\\
(\log_{2}n)^{1/2},& {q=\infty},
\end{array}
\right.
\end{equation}
where $1/q+1/q'=1$. Let $\{s_n\}$ be the sequence $\{d_n\}$ or  $\{d^n\}$ and let $f:\mathbb{N}\rightarrow \mathbb{R}$ be a positive and increasing function (for large $l \in \mathbb{N}$) such that, there is a constant $c>0$ and $f\left(2^j\right) \leq c f\left(2^{j-1}\right)$, for all $j \in \mathbb{N}$.
Then, there is a constant $C>0$ such that, for all $n \in \mathbb{N}$ we have
\begin{equation}\label{finita3}
\sup_{1\leq l\leq n} f(l)e_{l}(A,X) \leq C \sup_{1\leq l\leq n} f(l)s_{l}(A,X)
\end{equation}
(see \cite[Theorem 1]{Bernd2}, \cite[1.3.3]{Edmunds} and \cite[12.1.8]{Pietsch}). Consider the function $f$ defined for $l\in \mathbb{N}$, $l\geq 2$, by
$$
f(l)=  l^{\gamma/d}(\log_{2}l)^{\xi}  \left\{
\begin{array}{ll}
q^{-1/2},&{2\leq q< \infty},\\
(\log_{2}l)^{-1/2},& {q=\infty}.
\end{array}
\right. 
$$
This function $f$ satisfies the above condition and therefore from \ref{finita3}, \ref{finita1} and \ref{finita2} we have
$$
\sup_{1\leq l\leq n} f(l)e_{l} \left(\Lambda^{(1)}U_{2},L^{q}\right) 
\leq C \sup_{1\leq l\leq n} f(l)d_l \left(\Lambda^{(1)}U_{2},L^{q}\right) \leq \bar{C},
$$
$$
\sup_{1\leq l\leq n} f(l)e_{l} \left(\Lambda^{(1)}U_{q'},L^{2}\right) 
\leq C \sup_{1\leq l\leq n} f(l)d^l \left(\Lambda^{(1)}U_{q'}, L^{2}\right) \leq \bar{C},
$$
that is,
\begin{equation}\label{finita4}
e_{n}\left(\Lambda^{(1)}U_{2},L^{q}\right) \leq \bar{C} n^{-\gamma/d}(\log_{2}n)^{-\xi}
\left\{
\begin{array}{ll}
q^{1/2},&{2\leq q< \infty},\\
(\log_{2}n)^{1/2},& {q=\infty},
\end{array}
\right.
\end{equation}
\begin{equation}\label{finita5}
e_{n}\left(\Lambda^{(1)}U_{p},L^{2}\right) \leq \bar{C} n^{-\gamma/d}(\log_{2}n)^{-\xi}
\left\{
\begin{array}{ll}
(p-1)^{-1/2},&{1< p\leq 2},\\
(\log_{2}n)^{1/2},& {p=1}.
\end{array}
\right.
\end{equation}
Suppose $\gamma>d$ and consider the multiplier sequence $\{ \lambda^{(1)}_{\bf k}\}_{{\bf k} \in \mathbb{N}^d_{0}}$,
$\lambda^{(1)}_{\bf k}=\lambda^{(1)}(|{\bf k}|)$ where $\lambda^{(1)}$ is the function
$\lambda^{(1)}(t)= t^{-\gamma/2}(\log_{2} t)^{\xi/2}$ for $t>1$ and  $\lambda^{(1)}(t)=0$ for $0\leq t \leq 1$.
We denote by $R$ and $S$ the multiplier operators associated with the sequence $\{ \lambda^{(1)}_{\bf k}\}_{{\bf k} \in \mathbb{N}^d_{0}}$, $R:L^2\rightarrow L^q$ and $S:L^p\rightarrow L^2$, for $1\leq p \leq 2$ and $2 \leq q \leq \infty$. We have that the operator $\Lambda^{(1)}=R\circ S:L^p\rightarrow L^q$ is bounded. Applying the multiplicative property of entropy numbers (\ref{multiplicative}) and
(\ref{finita4}),  (\ref{finita5}) we get 
\begin{eqnarray*}
e_{2n-1}\left(\Lambda^{(1)}U_{p},L^{q}\right)&=&e_{2n-1}\left(\Lambda^{(1)}\right)=e_{2n-1}\left(R\circ S\right)\\
&\leq& e_{n}\left(R\right) e_{n}\left(S\right) = e_{n}\left(RU_{2},L^{q}\right) e_{n}\left(SU_{p},L^{2}\right)\\
&\leq& {\bar{C}}^2  n^{-\gamma/d}(\log_{2}n)^{-\xi} 
\left\{
\begin{array}{ll}
q^{1/2},&{2\leq q< \infty},\\
(\log_{2}n)^{1/2},& {q=\infty},
\end{array}
\right\} \cdot
\left\{
\begin{array}{ll}
(p-1)^{-1/2},&{1< p\leq 2},\\
(\log_{2}n)^{1/2},& {p=1},
\end{array}
\right\}  \\
&=& {\bar{C}}^2  n^{-\gamma/d}(\log_{2}n)^{-\xi}
\left\{
\begin{array}{ll}
q^{1/2}(p-1)^{-1/2},&\textbf{$1< p\leq 2$, $2\leq q<\infty$},\\
	q^{1/2}(\log_{2}n)^{1/2},&\textbf{$p=1$, $2\leq q<\infty$},\\
	(p-1)^{-1/2}(\log_{2}n)^{1/2},&\textbf{$1<p\leq 2$, $q=\infty$},\\
	\log_{2}n,&\textbf{$p=1$, $q=\infty$.}
	\end{array}
	\right.
\end{eqnarray*}
If $k\in \mathbb{N}$, $2n-1 \leq k \leq 2n+1$, then 
$e_{k}\left(\Lambda^{(1)}U_{p},L^{q}\right)\leq e_{2n-1}\left(\Lambda^{(1)}U_{p},L^{q}\right)$ and 
\linebreak $n^{-\gamma/d}(\log_{2}n)^{-\xi} \leq C_1 k^{-\gamma/d}(\log_{2}k)^{-\xi}$, $\log_{2}n \leq C_1\log_{2}k$ for
a constant $C_1>0$ and all $n,k \in \mathbb{N}$.
Hence (\ref{applicationfinite4}) follows from the above estimate.
\end{proof}

\begin{remark}\label{observacionentropiainfinitas}
	Let $\gamma,r\in\R$, $\gamma,r>0$. For $N, k\in\N$ and $n=\dim \mathcal{T}_{N}$, let 
	\[
A_{N,k}=-\frac{1}{n}\left(k\ln 2+\gamma \sum_{l=1}^{N}l^{r}d_{l} \right).
	\]
	 Since $d_{l}\asymp l^{d-1}$ and $n\asymp N^{d}$, then
	\[
	\sum_{l=1}^{N}l^{r}d_{l}\asymp\sum_{l=1}^{N}l^{r+d-1}\asymp \int_{0}^{N}x^{d+r-1}dx \asymp N^{d+r},
	\]
	and hence
	\begin{equation}\label{ecuacioninfinitasentropia1}
	A_{N,k}\asymp g(N),
	\end{equation}
	where $g(x)=-kx^{-d}-x^{r}$. The function $g$ assumes its absolute maximum at the point $x_{k}=(d/r)^{1/(d+r)}k^{1/(d+r)}$. Since for all $r>0$ and $x>1$, we have $-d  x^{-d-1}k\leq k[(x+1)^{-d}-x^{-d}]\leq 0 $ and $0\leq (x+1)^{r}-x^{r}\leq r x^{r-1}$, then $g(x+1)\geq  g(x)-rx^{r-1}$.	Thus, since $g$  is decreasing for $x>x_{k}$, it follows that
	\[
	g(x)-rx^{r-1}\leq g(x+1)\leq g(x+t)\leq g(x), \hspace{2ex} x\geq x_{k},\hspace{1ex}  0\leq t \leq 1.
	\]
	If $\overline{N}\in\N$, $x_{k}\leq \overline{N}\leq x_{k}+1$, then 
	\[
		g(x_{k})-rx_{k}^{r-1}\leq g(\overline{N})\leq \sup_{N}g(N)\leq g(x_{k})
		\]
	and thus
	\[
	-C_{1}k^{r/(d+r)}-C_{2}k^{(r-1)/(d+r)}\leq \sup_{N}g(N)\leq -C_{1}k^{r/(d+r)},
	\]
	with $C_{1}=(d+r)d^{-d/(d+r)}r^{-r/(d+r)}$ and $C_{2}=r^{(d+1)/(d+r)}d^{(r-1)/(d+r)}$. Therefore, from (\ref{ecuacioninfinitasentropia1}) we get 
	\[
	\sup_{N}A_{N,k}\asymp \sup_{N}g(N)\asymp k^{r/(d+r)}.
	\]
\end{remark}
\begin{proof}[\textup{\textbf{Proof of Theorem \ref{teoremainfinitasentropia}}.}]
	First we will prove (\ref{ecuacioninfinitasentropia2}). Note that
	\begin{equation*}
	2^{-k/n}\left(\prod_{l=1}^{N}\abs{\lambda(l)}^{d_{l}} \right)^{1/n}= e^{-\left(k\ln 2+\gamma \sum_{l=1}^{N}l^{r} d_{l}\right)/n}=e^{A_{N,k}},
	\end{equation*}
	where $n=\dim \mathcal{T}_{N}$. Then by Theorem \ref{teoremainferioressentropia} it follows that
	\begin{equation}\label{ecuacioninfinitasentropia3}
	e_{k}\left(\Lambda^{(2)}U_{p},L^{q}\right)\gg e^{A_{N,k}}\mathcal{V}_{n}.
	\end{equation}
	From Proposition \ref{proposicionaplicaciones1} we get 
	\[
	\sum_{l=1}^{N}l^{r}d_{l}\leq \sum_{l=1}^{N}\left(\frac{2^{1-d}\pi^{d/2}}{\Gamma(d/2)}l^{d+r-1}+C_{1}l^{\theta+r}\right).
	\]
    Taking  $f(x)=(2^{1-d}\pi^{d/2}/\Gamma(d/2)) x^{d+r-1}+C_{1}x^{\theta+r}$ we get 
	\begin{equation*}
		\sum_{l=1}^{N}l^{r}d_{l}\leq \sum_{l=1}^{N}f(l)\leq \int_{1}^{N+1}f(x)dx\leq \frac{2^{1-d}\pi^{d/2}}{\Gamma(d/2)(d+r)}(N+1)^{d+r}+C_{2}(N+1)^{\theta+r+1}.
	\end{equation*}
   	Again, by Proposition \ref{proposicionaplicaciones1} we have that $n\geq \bigl(2^{1-d}\pi^{d/2}/d\Gamma(d/2)\bigl)N^{d}$ and therefore
	\begin{align*}
	\frac{\gamma}{n}\sum_{l=1}^{N}l^{r}d_{l}&\leq \frac{d\Gamma(d/2)\gamma}{2^{1-d}\pi^{d/2}N^{d}}\left(\frac{2^{1-d}\pi^{d/2}}{\Gamma(d/2)(d+r)}(N+1)^{d+r}+C_{2}(N+1)^{\theta+r+1} \right)\\
	&=\frac{d\gamma}{(d+r)}\left(1+\frac{1}{N}\right)^{d}(N+1)^{r}+C_{3}\left(1+\frac{1}{N}\right)^{d}(N+1)^{\theta+r+1-d}.
	\end{align*}
	But $(1+1/N)^{d}=1+b_{1}/N+b_{2}/N^{2}+\cdots+1/N^{d}\leq (1+C_{4}/N)$, then
	\begin{equation*}
	\frac{\gamma}{n}\sum_{l=1}^{N}l^{r}d_{l}\leq \frac{d\gamma}{(d+r)}(N+1)^{r}+C_{5}\left(1+\frac{1}{N}\right)^{r}N^{r-1}+C_{6}(N+1)^{\theta+r+1-d}\leq \frac{d\gamma}{(d+r)}(N+1)^{r}+C_{7},
	\end{equation*}
if $0\leq r\leq d-\theta-1$. For $d\geq 4$, $\theta=d-2$ and thus we must have $0<r\leq 1$, for $d=2$, $\theta=132/208$ and thus $0<r\leq 76/208$, and for $d=3$ we must have $0<r\leq 11/16$.
Since	$(N+1)^{r}=N^{r}( 1+N^{-1})^{r}\leq N^{r}+C_{r}N^{r-1}\leq N^{r}+C_{8}$, we get
	\begin{equation}\label{ecuacioninfinitasentropia4}
	\frac{\gamma}{n}\sum_{l=1}^{N}l^{r}d_{l}\leq \frac{d\gamma}{d+r}(N^{r}+C_{8})+C_{7}=\frac{d\gamma}{d+r}N^{r}+C_{9}.
	\end{equation}
	From Proposition \ref{proposicionaplicaciones1} and (\ref{ecuacioninfinitasentropia4}), we obtain
	\begin{equation}\label{ecuacioninfinitasentropia6}
A_{N,k}\geq -\frac{d\Gamma(d/2)(\ln 2)k}{2^{1-d}\pi^{d/2}}N^{-d}-\frac{d\gamma}{d+r}N^{r}-C_{9}.
	\end{equation}
	Consider the function
	\[
		g(x)= -\frac{d\Gamma(d/2)(\ln 2)k}{2^{1-d}\pi^{d/2}}x^{-d}-\frac{d\gamma}{d+r}x^{r}-C_{9}.
	\]
	The absolute maximum value of the function $g$ is assumed at the point
	\[
	x_{k}=\left(\frac{(d+r)d\Gamma(d/2)(\ln 2)}{2^{1-d}r\gamma\pi^{d/2}} \right)^{1/(d+r)}k^{1/(d+r)}.
		\]
	We can show, as in Remark \ref{observacionentropiainfinitas}, that there is a constant $C_{10}$ such that 
	\[
	g(x)-C_{10}\leq g(x+t)\leq g(x), \hspace{2ex} x\geq x_{k}, \hspace{2ex} 0\leq t \leq 1
	\]
	and hence  
	\[
	g(x_{k})-C_{10}\leq \sup_{N}g(N)\leq g(x_{k}).
	\]
	 Therefore, it follows from (\ref{ecuacioninfinitasentropia6}) that
	\[
	\sup_{N}A_{N,k}\geq \sup_{N}g(N)\geq g(x_{k})-C_{10}.
	\]
	But 
	\[
	g(x_{k})=-\gamma^{d/(d+r)}\biggl( \frac{(d+r)2^{d-1}d\Gamma(d/2)(\ln 2)}{r\pi^{d/2}}\biggl)^{r/(d+r)}k^{r/(d+r)}-C_{9}\\
		=-\mathcal{C}k^{r/(d+r)}-C_{9},
	\]
	 and thus 
	 \begin{equation}\label{ecu414}
	 \sup_{N}A_{N,k}\geq -\mathcal{C}k^{r/(d+r)}-C_{11}.
	 \end{equation}
	Then, by (\ref{ecuacioninfinitasentropia3}) 
	\[
	e_{k}\left(\Lambda^{(2)}U_{p},L^{q}\right)\gg e^{-\mathcal{C}k^{r/(d+r)}}\mathcal{V}_{n}.
	\]
 Let $\overline{N}\in\N$, $\overline{N}\in(x_{k}-1,x_{k}+1)$ such that $\sup_{N}g(N)=g(\overline{N})$. From Remark \ref{observacionentropiainfinitas}, $\overline{N}\asymp k^{1/(d+r)}$, thus $\overline{n}=\dim \mathcal{T}_{\overline{N}}\asymp\overline{N}^{d}\asymp k^{d/(d+r)}$, and hence $\log_{2}n\asymp  \log_{2}k$. Therefore
	\[
	e_{k}\left(\Lambda^{(2)}U_{p},L^{q}\right)\gg e^{-\mathcal{C}k^{r/(d+r)}}\mathcal{V}_{k},
	\]
	concluding the proof of (\ref{ecuacioninfinitasentropia2}).

 Now, we will prove (\ref{ecuacioninfentropia}). From Remark \ref{Remark36} we have that $\Lambda^{(2)}\in K_{\epsilon,2}$ for $0<\epsilon< [2r-\delta(d-r)]/r$ and any $0< \delta< 2r/(d+r)$.	
 From Remark \ref{Remark36}, there is a positive constant $C_{12}$ such that, $\theta_{N_{k},N_{k+1}}\leq C_{12}2^{(\delta(d-r)/r)k}N^{d-r}$ for all $k,N \in \N$ and since	 $M\leq \epsilon^{-1}\log_{2}\theta_{N_{1},N_{2}}\leq C_{13}\log_{2}N$, then  $
	\log_{2}\theta_{N_{k},N_{k+1}}\leq C_{14}\log_{2}N\leq C_{15}\log_{2}n$ for $1\leq k\leq M$.
	Therefore
	\begin{equation}\label{supinfinitasentropia1}
	\sup_{1\leq k\leq M}\left(\log_{2}\theta_{N_{k},N_{k+1}}\right)^{1/2}\leq C_{16}(\log_{2}n)^{1/2}.
	\end{equation}
	For $2\leq q \leq \infty$, we obtain from (\ref{ecuacionentropiasup2})
	\begin{align*}
	\chi_{k}&\leq C_{17} \sup_{N\geq 1}\left(2^{-k}\right)^{1/n}\left(\prod^{N}_{j=1}\abs{\lambda(j)}^{d_{j}} \right)^{1/n}\left\{
	\begin{array}{ll}
	q^{1/2},& 2\leq q<\infty,\\
	(\log_{2}n)^{1/2},& q=\infty,
	\end{array}
	\right.\\
	&=C_{17}\sup_{N}e^{A_{N,k}}\left\{
	\begin{array}{ll}
	q^{1/2},& 2\leq q<\infty,\\
	(\log_{2}n)^{1/2},& q=\infty.
	\end{array}
	\right.
	\end{align*}
	From (\ref{ecuacionentropiasup2}) we also have $\chi_{k}\geq C_{18}\sup_{N}e^{A_{N,k}}$ and since $n\asymp N^{d}$, then
	\begin{equation}\label{supinfinitasentropia2}
		e^{\sup_{N}A_{N,k}}\ll \chi_{k}\ll \left\{
	\begin{array}{ll}
	e^{\sup_{N}A_{N,k}}q^{1/2},& 2\leq q<\infty,\\
	e^{\sup_{N}\left( A_{N,k}+ (\ln(\log_{2}N)) /2\right)},& q=\infty.
	\end{array}
	\right.
	\end{equation}
Consider $2\leq q <\infty$. If $f(x)=(2^{1-d}\pi^{d/2}/\Gamma(d/2))x^{d+r-1}-Cx^{\theta+r}$, then by Proposition \ref{proposicionaplicaciones1} it follows that
	\begin{equation}\label{supinfinitasentropia3}
	\sum_{l=1}^{N}l^{r}d_{l} \geq \sum_{l=1}^{N}f(l)\geq \int_{0}^{N}f(x)dx=\frac{2^{1-d}\pi^{d/2}}{\Gamma(d/2)(d+r)}N^{d+r}-\frac{C}{\theta+r+1} N^{\theta+r+1}.
	\end{equation}
The supremum $\sup_{N}A_{N,k}$ is assumed in $N\in \N$ satisfying	 $N\asymp k^{1/(d+r)}$ (see Remark \ref{observacionentropiainfinitas}). Again, from Proposition \ref{proposicionaplicaciones1} and since $0<r\leq 1$, we get
	\begin{equation}\label{supinfinitasentropia4}
	\frac{k}{n}\geq \frac{k}{F N^{d}}-\frac{\overline{C}k}{F^{2}N^{d+1}}\geq \frac{k}{FN^{d}}-C_{19}k^{(r-1)/(d+r)}\geq \frac{d \Gamma(d/2)k}{2^{1-d}\pi^{d/2}}N^{-d}-C_{19}.
	\end{equation}
Now, using Proposition \ref{proposicionaplicaciones1} and (\ref{supinfinitasentropia3}), 
	\begin{equation}\label{supinfinitasentropia5}
	\frac{\gamma}{n}\sum_{l=1}^{N}l^{r}d_{l}\geq \biggl(\frac{\gamma}{F N^{d}} -\frac{\overline{C}\gamma}{F^{2}N^{d+1}}\biggl)\biggl(\frac{2^{1-d}\pi^{d/2}}{(d+r)\Gamma(d/2)}N^{d+r}-\frac{C}{\theta+r+1}N^{\theta+r+1} \biggl) \geq\frac{d\gamma}{(d+r)}N^{r}-C_{20},
	\end{equation}
	since $0<r\leq d-1-\theta$. Hence, from (\ref{supinfinitasentropia4}) and (\ref{supinfinitasentropia5}) we get
	\begin{equation}\label{supinfinitasentropia6}
A_{N,k}=-\frac{k}{n}\ln 2 -\frac{\gamma}{n}\sum_{l=1}^{N}l^{r}d_{l}\leq- \frac{d\Gamma(d/2)(\ln 2)k}{2^{1-d}\pi^{d/2}}N^{-d}-\frac{d\gamma}{d+r}N^{r}+C_{21}.
	\end{equation}
	Consider the function $g_{1}(x)= g(x)+C_{9}+C_{21}$. The absolute maximum value of  $g_{1}$ is assumed at the  point $
	x_{k}$ 	and therefore
	\begin{equation*}
	\sup_{N}A_{N,k}\leq \sup_{N}g_{1}(N)\leq g_{1}(x_{k})=-\mathcal{C} k^{r/(d+r)}+C_{21}.
	\end{equation*}
	Then, by (\ref{supinfinitasentropia2}),
	\begin{equation}\label{supinfinitasentropia7}
	\chi_{k}\ll e^{-\mathcal{C}k^{r/(d+r)}}q^{1/2},\hspace{3ex}  2\leq q<\infty.
	\end{equation}
	For $q=\infty$, from  (\ref{supinfinitasentropia2}) and (\ref{supinfinitasentropia6}) we obtain
	\begin{equation*}
	\chi_{k}\ll e^{\sup_{N}\left( A_{N,k}+(\ln(\log_{2}N))/2\right)}\leq \sup_{N} e^{-\frac{d\Gamma(d/2)(\ln 2)k}{2^{1-d}\pi^{d/2}}N^{-d}-\frac{d\gamma}{d+r}N^{r}+\frac{1}{2}\ln(\log_{2}N)+C_{21}}.
	\end{equation*}
Consider now the function $g_{2}(x)=g_{1}(x)+\left(\ln(\log_{2}x)\right)/2$.
	As in Lemma \ref{lemafinitaentropia}, we can show that the absolute maximum value of $g_{2}$ is assumed at a point $\overline{x}_{k}$ satisfying $\overline{C}_{1}x_{k}\leq \overline{x}_{k}\leq \overline{C}_{2}x_{k}$ and hence $\sup_{N}g_{2}(N)$ is obtained when $N\asymp \bar{x}_{k}\asymp x_{k}\asymp k^{1/(d+r)}$. Therefore, for $q=\infty$, we have that
	\begin{equation}\label{ecuentrop1}
	\chi_{k}\ll \sup_{N}\left( e^{A_{N,k}}(\log_{2}N)^{1/2}\right)\ll \sup_{N}e^{A_{N,k}}\left( \log_{2}k^{1/(d+r)}\right)^{1/2}\ll e^{-\mathcal{C} k^{r/(d+r)}}(\log_{2}k)^{1/2},
	\end{equation}
	and by (\ref{ecu414}), (\ref{supinfinitasentropia2}), (\ref{supinfinitasentropia7}) and (\ref{ecuentrop1}) it follows that
	\begin{equation}\label{supinfinitasentropia8}
	e^{-\mathcal{C}k^{r/(d+r)}}\ll\chi_{k}\ll e^{-\mathcal{C}k^{r/(d+r)}}\left\{
	\begin{array}{ll}
	q^{1/2},& 2\leq q<\infty,\\
	(\log_{2}k)^{1/2},& q=\infty.
	\end{array}
	\right.
	\end{equation}

Fixed $k\in\N$, let $N\in\N$ such that
	\[\left(\frac{\mathcal{C}}{\gamma}\right)^{1/r}k^{1/(d+r)}\leq N <\left(\frac{\mathcal{C}}{\gamma}\right)^{1/r}k^{1/(d+r)}+1 .\]
		Then 
    \[
    \mathcal{C}k^{r/(d+r)}\leq \gamma N^{r}\leq\mathcal{C}k^{r/(d+r)} +C'k^{(r-1)/(d+r)}\leq \mathcal{C}k^{r/(d+r)}+C' 
    \]
	and thus 
	\begin{equation}\label{ecu424}
	    \overline{C}e^{-\mathcal{C}k^{r/(d+r)}}\leq \lambda(N)\leq e^{-\mathcal{C}k^{r/(d+r)}}
	\end{equation}
	for positive constants $C'$ and $\overline{C}$. Since by (\ref{supinfinitasentropia8}) $\chi_{k}\asymp e^{-\mathcal{C} k^{r/(d+r)}}$, there are positive constants $D$ and $\overline{D}$ such that $D\chi_{k}\leq \lambda(N)\leq \overline{D}\chi_{k} $. Let $b\in\N_{0}$ such that $D\geq 2^{-b}$. Therefore we have $\lambda(N)\geq 2^{-b}\chi_{k}$ 	and since $\Lambda^{(2)}\in K_{\epsilon,2}$, for some $\epsilon>0$, applying Theorem \ref{teoremasuperioresentropia} we obtain
		\begin{align*}
		e_{\eta+bn}\left(\Lambda^{(2)} U_{p},L^{q}\right)&\ll \abs{\lambda(N)}\left(\left\{
		\begin{array}{ll}
			q^{1/2},&\textbf{$2\leq p\leq \infty$, $2\leq q<\infty$},\\
			\sup_{1\leq j\leq M}(\log_{2}\theta_{N_{j},N_{j+1}})^{1/2},&\textbf{$2\leq p\leq \infty$, $q=\infty$},
		\end{array}
		\right.\right. \nonumber\\
		&+   \left.\sum_{j=M+1}^{\infty}2^{-j}\theta^{1/2-1/q}_{N_{j},N_{j+1}}\right).
	\end{align*}
     But from Remark \ref{Remark36}, $\sum_{k=M+1}^{\infty}2^{-k} \theta_{N_{k},N_{k+1}}^{1/p-1/q} \ll 1$ and $\log_{2}\theta_{N_{j},N_{j+1}}\ll\log_{2}n$ and since $\log_{2}n\asymp \log_{2}N^{d}\asymp \log_{2}k^{d/(d+r)}\asymp \log_{2}k$, hence by (\ref{ecu424}) 		\[
		e_{\eta+bn}\left(\Lambda^{(2)} U_{p},L^{q}\right)\ll e^{-\mathcal{C} k^{r/(d+r)}}\left\{
	\begin{array}{ll}
	q^{1/2},&\textbf{$2\leq p\leq \infty$, $2\leq q<\infty$},\\
	(\log_{2}k)^{1/2},&\textbf{$2\leq p\leq \infty$, $q=\infty$}.
	\end{array}
	\right.
	\]
	Considering the result of Theorem \ref{teoremasuperioresentropia} for $1\leq q <2$, we obtain
	\begin{equation}\label{supinfinitasentropia9}
 e_{\eta+bn}\left(\Lambda^{(2)} U_{p},L^{q}\right)\ll e^{-\mathcal{C} k^{r/(d+r)}}\left\{
	\begin{array}{ll}
	q^{1/2},&\textbf{$2\leq p\leq \infty$, $1\leq q<\infty$},\\
	(\log_{2}k)^{1/2},&\textbf{$2\leq p\leq \infty$, $q=\infty$}.
	\end{array}
	\right.
	\end{equation}
	From Remark \ref{observacion-superiores-1} and remembering that $N\asymp k^{1/(d+r)}$, we get
	\begin{equation*}
	\eta+bn=k+bn+\sum_{j=1}^{M}m_{j}\leq k+C_{22}N^{d}+C_{23}N^{d-r}\leq k+C_{24}k^{d/(d+r)}.
	\end{equation*}
	Now, from properties of entropy numbers and (\ref{supinfinitasentropia9}), we obtain
	\begin{equation}\label{supinfinitasentropia10}
	e_{[k+C_{24}k^{d/(d+r)}]}\left(\Lambda^{(2)} U_{p},L^{q}\right)\ll e^{-\mathcal{C} k^{r/(d+r)}}\left\{
	\begin{array}{ll}
	q^{1/2},&\textbf{$2\leq p\leq \infty$, $1\leq q<\infty$},\\
	(\log_{2}k)^{1/2},&\textbf{$2\leq p\leq \infty$, $q=\infty$}.
	\end{array}
	\right.
	\end{equation}
	Let $\xi_{k}=k+C_{24}k^{d/(d+r)}$. Then
	\begin{align*}
	\mathcal{C}\left(-k^{r/(d+r)}+\xi^{r/(d+r)}_{k}\right)	&=\mathcal{C} k^{r/(d+r)}\left( \frac{rC_{24}}{(d+r)}k^{-r/(d+r)}-\frac{rd C_{24}^{2}}{2(d+r)^{2}}k^{-2r/(d+r)}+\cdots\right)\\
	&\leq \mathcal{C} \frac{rC_{24}}{(d+r)}= C_{25}
	\end{align*}
	and therefore 
	\[
	e^{-\mathcal{C} k^{r/(d+r)}}\leq e^{C_{25}}e^{-\mathcal{ C}\xi_{k}^{r/(d+r)}}\ll e^{-\mathcal{C}\xi_{k}^{r/(d+r)}}.
	\]
	 Then from (\ref{supinfinitasentropia10}) 
	\begin{equation}\label{supinfinitasentropia11}
		e_{[\xi_{k}]}\left(\Lambda^{(2)} U_{p},L^{q}\right)\ll e^{-\mathcal{C} \xi_{k}^{r/(d+r)}}\left\{
	\begin{array}{ll}
	q^{1/2},&\textbf{$2\leq p\leq \infty$, $1\leq q<\infty$},\\
	(\log_{2}k)^{1/2},&\textbf{$2\leq p\leq \infty$, $q=\infty$}.
	\end{array}
	\right.
	\end{equation}
	We observe that
	\begin{equation*}
		\xi_{k}^{r/(d+r)}=k^{r/(d+r)}\biggl (1+\frac{r C_{24}}{d+r}k^{-r/(d+r)} -\frac{rdC_{24}^{2}}{2(d+r)^{2}}k^{-2r/(d+r)}+\cdots\biggl),
	\end{equation*}
	and hence
	\[
\xi_{k}^{r/(d+r)}\geq k^{r/(d+r)}\left (1+\frac{r C_{24}}{d+r}k^{-r/(d+r)} -\frac{rdC_{24}^{2}}{2(d+r)^{2}}k^{-2r/(d+r)}\right)
	\]
	and
	\[
	\xi_{k+1}^{r/(d+r)}\leq (k+1)^{r/(d+r)}\biggl (1+\frac{r C_{24}}{d+r}(k+1)^{-r/(d+r)} \biggl).
	\]
	Since $0<(k+1)^{r/(d+r)}-k^{r/(d+r)}\leq r/(d+r)$, it follows that
	\begin{equation*}
	0\leq \xi_{k+1}^{r/(d+r)}-\xi_{k}^{r/(d+r)}
		\leq (k+1)^{r/(d+r)}-k^{r/(d+r)}+C_{26}\leq C_{27}
	\end{equation*}
	and thus
	\begin{equation}\label{supinfninitasentropia14}
	1\leq \frac{e^{-\mathcal{C}\xi_{k}^{r/(d+r)}}}{e^{-\mathcal{C}\xi_{k+1}^{r/(d+r)}}}\leq C_{28}.
	\end{equation}
	Finally, let $l$ an integer satisfying $[\xi_{k}]\leq l\leq [\xi_{k+1}]$. Then from properties of entropy numbers, from (\ref{supinfinitasentropia11}) and (\ref{supinfninitasentropia14}) we obtain
	\begin{align*}
	e_{l}\left(\Lambda^{(2)}U_{p},L^{q}\right)&\leq e_{[\xi_{k}]}\left(\Lambda^{(2)}U_{p},L^{q}\right)\ll e^{-\mathcal{C} \xi_{k}^{r/(d+r)}}\left\{
		\begin{array}{ll}
			q^{1/2},&\textbf{$2\leq p\leq \infty$, $1\leq q<\infty$},\\
			(\log_{2}k)^{1/2},&\textbf{$2\leq p\leq \infty$, $q=\infty$}.
		\end{array}
		\right.\\
	&\leq e^{-\mathcal{C} l^{r/(d+r)}}\left\{
		\begin{array}{ll}
			q^{1/2},&\textbf{$2\leq p\leq \infty$, $1\leq q<\infty$},\\
			(\log_{2}l)^{1/2},&\textbf{$2\leq p\leq \infty$, $q=\infty$},
		\end{array}
		\right.
	\end{align*}
	concluding the proof of (\ref{ecuacioninfentropia}).
\end{proof}

\begin{remark}\label{nuevaobser}
 Let $B_R^{*}=\{\textup{\textbf{x}}\in \mathbb{R}^d : \abs{\textup{\textbf{x}}}_{*} \leq R\}$ and  $N_d^{*}$  
the number of points of integer coordinates contained in the ball $B_R^{*}$. 
We have that $N_d^{*}(l)=(2l+1)^{d}$ and $\# A_l^{*}=(l+1)^{d}$. Then there is a constant $C$, such that, for all $l, N \in \mathbb{N}$
$$
dl^{d-1} \leq d_l^{*} \leq dl^{d-1} + Cl^{d-2}
\hspace{2ex}\textup{and}\hspace{2ex}
N^d \leq \dim{{\cal T}_N^*}\leq N^d + C N^{d-1}.
$$

Considerer the multiplier operators $\Lambda_{*}^{(1)}$ and $\Lambda_{*}^{(2)}$ defined in the introduction. We note that the Theorem \ref{teorema-media-1} was proved in \cite{Sergio} also for operators of this type. Through minor modifications in the proofs in this paper we can show that Theorem \ref{teoremainferioressentropia} and Theorem \ref{teoremasuperioresentropia} also hold if we change $\mathcal{H}_{l}$, $\mathcal{T}_{N}$, $d_{l}$ and $\lambda_{\textup{\textbf{k}}}$ by  $\mathcal{H}^{*}_{l}$, $\mathcal{T}^{*}_{N}$, $d^{*}_{l}$ and $\lambda^{*}_{\textup{\textbf{k}}}$ and Theorem \ref{aplicacionesentropiafinitas1} also hold if we change $\Lambda^{(1)}$ by $\Lambda_{*}^{(1)}$ and Theorem \ref{teoremainfinitasentropia}  hold if we change $\Lambda^{(2)}$ by $\Lambda_{*}^{(2)}$ and the constant $\mathcal{C}$ by the constant $\mathcal{C}_{*}$ given in the introduction. The Theorem \ref{teoremainfinitasentropia} is valid in this case for $0<r \leq 1$, for all $d\in \mathbb{N}$.
\end{remark}

\begin{remark} \label{examples}
\label{remark2} Consider the real $l$-dimensional unit sphere ${\mathbb{S}}
^{l}$ in $\mathbb{R}^{l+1}$, that is, the set of all $x\in \mathbb{R}^{l+1}$
such that $x_1^2 +\cdots + x_{l+1}^2=1$, endowed with the normalized
Lebesgue measure $dx$. Walsh functions on ${\mathbb{S}}^{l}$ were introduced
in \cite{Bordin}.
The set $\left\{\zeta_{n}\right\}_{n\in \mathbb{N}}$ of the Walsh functions on ${\mathbb{S}}^{l}$ is a complete
orthonormal subset of real-valued functions of $L^{2} \left(\mathbb{S}^{l}\right)$
and $|\zeta_{n}(x)|=1$ for all $n\in \mathbb{N}$ and $x \in \mathbb{S}^{l} $ (see \cite{Bordin,Sergio}).

Now, consider the interval $[0, 2\pi ] $ with the normalized Lebesgue measure $(1/(2\pi)) dt$ and for each $x \in [0, 2\pi]$ let
$\varphi_{0} (x) = 1$, $\varphi_{2k} (x) = {\sqrt{2}}\cos (kx)$, $\varphi_{2k-1}(x) = \sqrt{2}\sin (kx)$, for all $k \in \mathbb{N}$.
The trigonometric system $\{\varphi_{m}\}_{m\in \N_{0}}$ is a complete orthonormal subset of $L^{2}([0, 2\pi])$, of real-valued functions, uniformly bounded in \linebreak$L^{\infty}([0,2\pi])$.

For each $1\leq j\leq d$, let $\left\{\phi_{l}^{(j)}\right\}_{l\in \N_{0}}$ be a Vilenkin system 
$\left\{ Z_{n} \right\}_{n\in\N_{0}}$  of $L^{2}\left(\G\right)$, or the Walsh system of $L^{2} \left(\mathbb{S}^{l}\right)$ 
for some $l \geq 2$, or the trigonometric system $\{\varphi_{m}\}_{m\in \N_{0}}$ of $L^{2}([0, 2\pi])$. Consider the complete 
orthonormal system $\left\{\phi_{\textup{\textbf{m}}}\right\}_{\textup{\textbf{m}}\in \N_{0}^{d}}$ of $L^{2}(\Omega)$ 
as in the Proposition \ref{proposi1}. Then all the results in this paper are true for this system.
\end{remark}


\end{document}